\documentclass[10pt, twosides, a4paper]{article}

\usepackage{theorem,makeidx,float,calc,verbatim}
\usepackage{latexsym,amscd}
\usepackage{amsmath,amsfonts,amssymb,mathrsfs}
\usepackage{enumerate,qsymbols,euscript}

\input cyracc.def

\input xy
\xyoption{all}
\newdir{ >}{{}*!/-5pt/@{>}}
\newtheorem{theorem}{Theorem}[section]
\newtheorem{lemma}[theorem]{Lemma}
\newtheorem{proposition}[theorem]{Proposition}
\newtheorem{corollary}[theorem]{Corollary}
\theorembodyfont{\upshape}

\newtheorem{remark}[theorem]{Remark}

\newenvironment{proof}{%
\noindent{\it Proof.}\hskip 10 pt%
}{%
{\quad} \hfill$\Box$\\%
}

\def\og{\leavevmode\raise.3ex\hbox{$\scriptscriptstyle\langle\!\langle$~}}
\def\fg{\leavevmode\raise.3ex\hbox{~$\!\scriptscriptstyle\,\rangle\!\rangle$}}
\DeclareMathOperator{\des}{des}
\DeclareMathOperator{\rang}{rk}
\DeclareMathOperator{\Spec}{Spec}

\DeclareMathOperator{\Pic}{Pic}
\DeclareMathOperator{\pr}{pr}
\renewcommand{\det}{\mathrm{det}}
\DeclareMathOperator{\Id}{Id}
\DeclareMathOperator{\End}{End}

\normalsize
 \oddsidemargin
= 10 pt \evensidemargin = 32 pt

\begin{document}

\title{Maximal slope of tensor product of Hermitian vector bundles}
\author{{\sc Chen} Huayi\thanks{CMLS, Ecole Polytechnique, Palaiseau 91120, France. (huayi.chen@polytechnique.org)}}
\maketitle
\begin{abstract}
We give an upper bound for the maximal slope
of the tensor product of several non-zero
Hermitian vector bundles on the spectrum of
an algebraic integer ring. By Minkowski's
First Theorem, we need to estimate the
Arakelov degree of an arbitrary Hermitian
line subbundle $\overline M$ of the tensor
product. In the case where the generic fiber
of $M$ is semistable in the sense of
geometric invariant theory, the estimation is
established by constructing (through the
classical invariant theory) a special
polynomial which does not vanish on the
generic fibre of $M$. Otherwise we use an
explicte version of a result of Ramanan and
Ramanathan to reduce the general case to the
former one.
\end{abstract}

\section{Introduction}\label{Sec:Introduction}

\hskip\parindent It is well known that on a
projective and smooth curve defined over a
field of characteristic $0$, the tensor
product of two semistable vector bundles is
still semistable. This result has been
firstly proved by Narasimhan and Seshadri
\cite{Nara_Se65} by using analytic method in
the complex algebraic geometry framework.
Then this result has been reestablished by
Ramanan and Ramanathan
\cite{Ramanan_Ramanathan} in purely algebraic
context, through the geometric invariant
theory. Their method is based on a result of
Kempf \cite{Kempf78}, which has also been
independently obtained by Rousseau
\cite{Rousseau78}, generalizing the
Hilbert-Mumford criterion \cite{Mumford94} of
semistability in the sense of geometric
invariant theory. By reformulating the
results of Kempf and Ramanan-Ramanathan,
Totaro \cite{Totaro96} (see also
\cite{deShalit} for a review) has given a new
proof of a conjecture due to Fontaine
\cite{Fontaine79}, which had been firstly
proved by Faltings \cite{Faltings89}
asserting that the tensor product of two
semistable admissible filtered isocristals is
still semistable.

Let us go back to the case of vector bundles.
Consider a smooth projective curve $C$
defined over a field $k$. For any non-zero
vector bundle $E$ on $C$, the {\it slope} of
$E$ is defined as the quotient of its degree
by its rank and is denoted by $\mu(E)$. The
{\it maximal slope} $\mu_{\max}(E)$ of $E$ is
the maximal value of slopes of all non-zero
subbundles of $E$. By definition,
$\mu(E)\le\mu_{\max}(E)$. We say that $E$ is
{\it semistable} if the equality
$\mu(E)=\mu_{\max}(E)$ holds. If $E$ and $F$
are two non-zero vector bundles on $C$, then
$\mu(E\otimes F)=\mu(E)+\mu(F)$. The result
of Ramanan-Ramanathan
\cite{Ramanan_Ramanathan} implies that, if
$k$ is of characteristic $0$, then the
equality holds for maximal slopes, i.e.,
$\mu_{\max}(E\otimes
F)=\mu_{\max}(E)+\mu_{\max}(F)$. When the
characteristic of $k$ is positive, this
equality is not true in general (see
\cite{Gieseker73} for a counter-example).
Nevertheless, there always exists a constant
$a$ which only depends on $C$ such that
\begin{equation}\label{Equ:encadrement de mu max cas geometric}\mu_{\max}(E)+\mu_{\max}(F)\le\mu_{\max}(E\otimes
F)\le
\mu_{\max}(E)+\mu_{\max}(F)+a.\end{equation}

Hermitian vector bundles play in Arakelov
geometry the role of vector bundles in
algebraic geometry. Let $K$ be a number field
and $\mathcal O_K$ be its integer ring. We
denote by $\Sigma_\infty$ the set of all
embeddings of $K$ into $\mathbb C$. A
Hermitian vector bundle $\overline E=(E,h)$
on $\Spec\mathcal O_K$ is by definition a
projective $\mathcal O_K$-module of finite
type $E$ together with a family of Hermitian
metrics
$h=(\|\cdot\|_\sigma)_{\sigma\in\Sigma_\infty}$,
where for any $\sigma\in\Sigma_\infty$,
$\|\cdot\|_\sigma$ is a Hermitian norm on
$E\otimes_{\mathcal O_K,\sigma}\mathbb C$,
subject to the condition that the data
$(\|\cdot\|_\sigma)_{\sigma\in\Sigma_\infty}$
is invariant by the complex conjugation. That
is, for any $e\in E$, $z\in\mathbb C$ and
$\sigma\in\Sigma_\infty$, we have
$\|e\otimes\overline{z}\|_{\overline\sigma}=\|e\otimes
z \|_\sigma$.

The (normalized) {\it Arakelov degree} of a
Hermitian vector bundle $\overline E$ of rank
$r$ on $\Spec\mathcal O_K$ is defined as
\[\widehat{\deg}_n\overline
E=\frac{1}{[K:\mathbb Q]}
\Big(\log\#(E/\mathcal O_Ks_1+\cdots+\mathcal
O_Ks_r )-\frac
12\sum_{\sigma\in\Sigma_\infty}\log\det(
\left<s_i,s_j\right>_\sigma)\Big),\] where
$(s_1,\cdots,s_r)$ is an arbitrary element in
$E^r$ which defines a basis of $E_K$ over
$K$. This definition does not depend on the
choice of $(s_1,\cdots,s_r)$. The function
$\widehat{\deg}_n$ is invariant by any finite
extension of $K$. That is, if $K'/K$ is a
finite extension and if
$E'=E\otimes_{\mathcal O_K}\mathcal O_{K'}$,
then $\widehat{\deg}_n(\overline
E')=\widehat{\deg}_n(\overline E)$. The {\it
slope} of a non-zero Hermitian vector bundle
$\overline E$ on $\Spec\mathcal O_K$ is
defined as the quotient
$\widehat{\mu}(\overline
E):=\widehat{\deg}_n(\overline E)/\rang(E)$.
For more details, see \cite{BostBour96},
\cite{Bost2001}, \cite{Chambert}.

We say that a non-zero Hermitian vector
bundle $\overline E$ is {\it semistable} if
the {\it maximal slope}
$\widehat{\mu}_{\max}(\overline E)$ of
$\overline E$, defined as the maximal value
of slopes of its non-zero Hermitian
subbundles, equals its slope. If $\overline
E$ is a non-zero Hermitian vector bundle on
$\Spec\mathcal O_K$, Stuhler \cite{Stuhler76}
and Grayson \cite{Grayson84} have proved that
there exists a unique Hermitian subbundle
$\overline E_{\des}$ of $\overline E$ having
$\widehat{\mu}_{\max}(\overline E)$ as its
slope and containing all Hermitian subbundle
$\overline F$ of $\overline E$ such that
$\widehat{\mu}(\overline
F)=\widehat{\mu}_{\max}(\overline E)$.
Clearly $\overline E$ is semistable if and
only if $\overline E=\overline E_{\des}$. If
it is {\bf not} the case, then $\overline
E_{\des}$ is said to be the Hermitian
subbundle which {\it destabilizes} $\overline
E$.

In a lecture at Oberwolfach, J.-B. Bost
\cite{Bost97} has conjectured that the tensor
product of two semistable Hermitian vector
bundles on $\Spec\mathcal O_K$ is semistable.
This conjecture is equivalent to the
assertion that for any non-zero Hermitian
vector bundles $\overline E$ and $\overline
F$ on $\Spec\mathcal O_K$,
\[\widehat{\mu}_{\max}(\overline
E\otimes\overline
F)=\widehat{\mu}_{\max}(\overline
E)+\widehat{\mu}_{\max}(\overline F).\] We
always have the inequality
$\widehat{\mu}_{\max}(\overline
E\otimes\overline
F)\geq\widehat{\mu}_{\max}(\overline
E)+\widehat{\mu}_{\max}(\overline F)$. But
the inverse inequality remains open. Several
special cases of this conjecture have been
proved. Some estimations of type
\eqref{Equ:encadrement de mu max cas
geometric} have been established with error
terms depending on the ranks of the vector
bundles and on the number field $K$. We
resume some known results on this conjecture.

\begin{enumerate}[1)]
\item By definition of the maximal slope,
if $\overline E$ is a non-zero Hermitian
vector bundle and if $\overline L$ is a
Hermitian line bundle, that is, a Hermitian
vector bundle of rank one, then
\[\widehat{\mu}_{\max}(\overline
E\otimes\overline
L)=\widehat{\mu}_{\max}(\overline
E)+\widehat{\deg}_n(\overline
L)=\widehat{\mu}_{\max}(\overline
E)+\widehat{\mu}_{\max}(\overline L).\] The
geometric counterpart of this equality is
also true for positive characteristic case.
\item De Shalit and Parzanovski \cite{deShalit_Parzan} have
proved that, if $\overline E$ and $\overline
F$ are two semistable Hermitian vector
bundles on $\Spec\mathbb Z$ such that $\rang
E+\rang F\le 5$, then $\overline
E\otimes\overline F$ is semistable.
\item In \cite{BostBour96} (see also
\cite{Graftieaux00}), using the comparison of
a Hermitian vector bundle to a direct sum of
Hermitian line bundles, Bost has proved that
\[\widehat{\mu}_{\max}(\overline E_1\otimes\cdots\otimes\overline E_n)
\le\sum_{i=1}^n\Big(\widehat{\mu}_{\max}(\overline
E_i)+3\rang E_i\log(\rang E_i)\Big)
\]
for any family of non-zero Hermitian vector
bundles $(\overline E_i)_{i=1}^n$ on
$\Spec\mathcal O_K$.
\item Recently, Bost and K\"unnemann
\cite{Bost_Kunnemann} have proved that, if
$K$ is a number field and if $\overline E$
and $\overline F$ are two non-zero Hermitian
vector bundles on $\Spec\mathcal O_K$, then
\[\widehat{\mu}_{\max}(\overline E\otimes\overline F)
\le\widehat{\mu}_{\max}(\overline
E)+\widehat{\mu}_{\max}(\overline
F)+\frac{1}{2}\big(\log\rang E+\log\rang
F\big)+\frac{\log|\Delta_K|}{2[K:\mathbb Q]},\] where
$\Delta_K$ is the discriminant of $K$.
\end{enumerate}

We state the main result of this article as
follows:
\begin{theorem}\label{Thm:main theorem}
Let $K$ be a number field and $\mathcal O_K$
be its integer ring. If $(\overline
E_i)_{i=1}^n$ is a family of non-zero
Hermitian vector bundles on $\Spec\mathcal
O_K$, then
\begin{equation}\label{Equ:main theorem}
\widehat{\mu}_{\max}(\overline
E_1\otimes\cdots\otimes\overline
E_n)\le\sum_{i=1}^n\Big( \widehat{\mu}_{\max}(\overline
E_i)+\log(\rang E_i)\Big).
\end{equation}
\end{theorem}

The idea goes back to an article of Bost
\cite{Bost94} inspired by Bogomolov
\cite{Raynaud81}, Gieseker \cite{Gieseker77}
and Cornalba-Harris \cite{Cornalba_Harris}.
In an article of Gasbarri \cite{Gasbarri00}
appears also a similar idea. By Minkowski's
First Theorem, we reduce our problem to
finding an upper bound for the Arakelov
degree of an arbitrary Hermitian line
subbundle $\overline M$ of $\overline
E_1\otimes\cdots\otimes\overline E_n$. In the
case where $M_K$ is semistable (in the sense
of geometric invariant theory) for the action
of
$\mathbb{GL}(E_{1,K})\times\cdots\times\mathbb{GL}(E_{n,K})$,
the classical invariant theory gives
invariant polynomials with coefficients in
$\mathbb Z$ whose Archimedian norms are
``small''. The general case can be reduced to
the former one using an explicit version of a
result of Ramanan-Ramanathan
\cite{Ramanan_Ramanathan}.

The structure of this article is as follows.
In the second section we fix the notation and
present some preliminary results. In the
third section we recall the first principal
theorem in classical invariant theory and
discuss some generalizations in the case of
several vector spaces. We then establish in
the fourth section an upper bound for the
Arakelov degree of a Hermitian line subbundle
with semistable hypothesis. The fifth section
is contributed to some basic notions for
filtrations in the category of vector spaces.
Then in the sixth section, we state an
explicit version of a result of
Ramanan-Ramanathan in our context and,
following the method of Totaro, give a proof
for it. In the seventh section is presented a
criterion of semistability (for Hermitian
vector bundles) which is an arithmetic
analogue of a result of Bogomolov. In the
eighth section, we explain how to use the
result in previous sections to reduce the
majoration of the Arakelov degree of an
arbitrary Hermitian line subbundle to the
case with semistability hypothesis, which has
already been discussed in the fourth section.
Finally, we give the proof of Theorem
\ref{Thm:main theorem} in the ninth section.

The result presented here is part of my
doctorial thesis \cite{Chen_these},
supervised by J.-B. Bost. The ideas in this
article are largely inspired by his article
\cite{Bost94} and his personal notes. I would
like to thank him deeply for his instruction
and his sustained encouragement. During my
visit to {\it Institut Joseph Fourier} in
Grenoble, E. Gaudron pointed out to me that
the method in this article, combined with his
recent result \cite{Gaudron07}, leads to an
estimation which is similar to
\eqref{Equ:main theorem} for the tensor
product of Adelic vector bundles. I am
grateful to him for discussions and for
suggestions. I would also like to express my
gratitude to the referee for his/her very
careful reading and for his/her numerous
useful suggestions to improve the writing of
this article.

\section{Notation and preliminary results}

\hskip\parindent Throughout this article, if
$K$ is a field and if $V$ is a vector space
of finite rank over $K$, we denote by
$\mathbb P(V)$ the $K$-scheme which
represents the functor
\begin{equation}\label{Equ:representable functor
projective space}
\begin{array}{ccc}
\mathbf{Schemes}/K&\longrightarrow&\mathbf{Sets}\\
(p:S\rightarrow\Spec
K)&\longmapsto&\Big\{\parbox{3.2cm}{\begin{center}locally
free quotient of rank $1$ of
$p^*V$\end{center}}\Big\}
\end{array}
\end{equation}
In particular, $\mathbb P(V)(K)$ classifies
all hyperplanes in $V$, or equivalently, all
lines in $V^\vee$. We denote by $\mathcal
O_V(1)$ the canonical line bundle on $\mathbb
P(V)$. In other words, if $\pi:\mathbb
P(V)\rightarrow\Spec K$ is the structural
morphism, then $\mathcal O_V(1)$ is the
quotient of $\pi^*V$ defined by the universal
object of the representable functor
\eqref{Equ:representable functor projective
space}. For any
integer $m\ge 1$, we use the expression
$\mathcal O_V(m)$ to denote the line bundle
$\mathcal O_V(1)^{\otimes m}$.

Let $G$ be an algebraic group over $\Spec K$
and $X$ be a projective variety over $\Spec
K$. Suppose that $G$ acts on $X$ and that $L$
is an ample $G$-linearized line bundle on
$X$. We say that a rational point $x$ of $X$
is {\it semistable} for the action of $G$
relatively to $L$ if there exists an integer
$D\ge 1$ and a section $s\in H^0(X,L^{\otimes
D})$ invariant by the action of $G$ such that
$x$ lies in the open subset of $X$ defined by
the non-vanishing of $s$. Clearly $x$ is
semistable for the action of $G$ relatively
to $L$ if and only if it is semistable for
the action of $G$ relatively to any strictly
positive tensor power of $L$.

In particular, if $G(K)$ acts linearly on a
vector space $V$ of finite rank over $K$,
then the action of $G$ on $V$ induces
naturally an action of $G$ on $\mathbb P(V)$,
and $\mathcal O_V(1)$ becomes a
$G$-linearized line bundle. Let $R$ be a
vector subspace of rank $1$ of $V^\vee$,
which is viewed as a point in $\mathbb
P(V)(K)$. Then $R$ is semistable for the
action of $G$ relatively to $\mathcal O_V(1)$
if and only if there exists an integer $m\ge
1$ and a non-zero section $s\in H^0(\mathbb
P(V),\mathcal O_V(m))=S^mV$ which is
invariant by the action of $G(K)$ such that
the composed homomorphism $\xymatrix{\relax
R\ar[r]&V^\vee\ar[r]^-{s}&K}$ is non-zero.

We present some estimations for maximal
slopes in geometric case. Let $k$ be an
arbitrary field and $C$ be a smooth
projective curve of genus $g$ defined over
$k$. Let
$b=\min\{\deg(L)\;|\;L\in\Pic(C),\;L\text{ is
ample}\}$ and let $a=b+g-1$.

\begin{lemma}
Let $E$ be a non-zero vector bundle on $C$.
If $H^0(C,E)=0$, then $\mu_{\max}(E)\le g-1$.
\label{Lem:annulation de H0 et borne de mu
max}
\end{lemma}
\begin{proof}
Since $H^0(C,E)=0$, for any non-zero
subbundle $F$ of $E$, we also have
$H^0(C,F)=0$. Recall that the Riemann-Roch
theorem asserts that
\[\rang_k H^0(C,F)-\rang_k H^1(C,F)=\deg(F)+\rang(F)(1-g).\]
Then $\deg(F)+\rang(F)(1-g)\le 0$, which
implies $\mu(F)\le g-1$. Since $F$ is
arbitrary, $\mu_{\max}(E)\le g-1$.
\end{proof}

\begin{proposition}
\label{Pro:encadrement de mu max otimes} For
any non-zero vector bundles $E$ and $F$ on
$C$, we have the inequality
\begin{equation*}
\mu_{\max}(E)+\mu_{\max}(F)\le\mu_{\max}(E\otimes
F )\le \mu_{\max}(E)+\mu_{\max}(F)+a,
\end{equation*}
where $a=b+g-1$ only depends on $C$.
\end{proposition}
\begin{proof}
1) Let $E_1$ be a subbundle of $E$ such that
$\mu(E_1)=\mu_{\max}(E)$ and let $F_1$ be a
subbundle of $F$ such that
$\mu(F_1)=\mu_{\max}(F)$. Since $E_1\otimes
F_1$ is a subbundle of $E\otimes F$, we
obtain
\[\mu_{\max}(E)+\mu_{\max}(F)=\mu(E_1)+\mu(F_1)
=\mu(E_1\otimes F_1)\le\mu_{\max}(E\otimes F
),\] which is the first inequality.

2) We first prove that, if $E'$ and $E''$ are
two non-zero vector bundles on $C$ such that
$\mu_{\max}(E')+\mu_{\max}(E'')<0$, then
$\mu_{\max}(E'\otimes E'')\le g-1$. In fact,
if $\mu_{\max}(E'\otimes E'')> g-1$, then by
Lemma \ref{Lem:annulation de H0 et borne de
mu max}, $H^0(C,E'\otimes E'')\neq 0$.
Therefore, there exists a non-zero
homomorphism $\varphi$ from ${E'}^\vee$ to
$E''$. Let $G$ be the image of $\varphi$,
which is non-zero since $\varphi$ is
non-zero. The vector bundle $G$ is a
subbundle of $E''$ and a quotient bundle of
${E'}^\vee$. Hence $G^\vee$ is a subbundle of
${E'}^{\vee\vee}\cong E'$. Therefore, we have
$\mu(G)\le\mu_{\max}(E'')$ and
$\mu(G^\vee)=-\mu(G)\le\mu_{\max}(E')$. By
taking the sum, we obtain
$\mu_{\max}(E')+\mu_{\max}(E'')\ge 0$.

We now prove the second inequality in the
proposition. By definition of $b$, there
exists a line bundle $M$ such that
$-b\le\mu_{\max}(E)+\mu_{\max}(F)+\deg(M)=
\mu_{\max}(E\otimes M)+\mu_{\max}(F)<0$.
Then, by combining the previously proved
result, we obtain $\mu_{\max}(E\otimes
M\otimes F)\le g-1$. Therefore,
\[\mu_{\max}(E\otimes F)\le g-1-\deg(M)\le
\mu_{\max}(E)+\mu_{\max}(F)+g+b-1.\]
\end{proof}

We now recall some classical results in
Arakelov theory, which will be useful
afterwards. We begin by introducing the
notation.

Let $\overline E$ be a Hermitian vector
bundle on $\Spec\mathcal O_K$. For any finite
place $\mathfrak p$ of $K$, we denote by
$K_{\mathfrak p}$ the completion of $K$ with
respect to $\mathfrak p$, equipped with the
absolute value $|\cdot|_{\mathfrak p}$ which
is normalized as $|\cdot|_{\mathfrak p
}=\#(\mathcal O_K/\mathfrak p)^{-v_{\mathfrak
p}(\cdot)}$ with $v_{\mathfrak p}$ being the
discrete valuation associated to $\mathfrak
p$. The structure of $\mathcal O_K$-module on
$E$ induces naturally a norm
$\|\cdot\|_{\mathfrak p}$ on $E_{K_{\mathfrak
p}}:=E\otimes_KK_{\mathfrak p }$ such that
$E_{K_{\mathfrak p}}$ becomes a Banach space
over $K_{\mathfrak p}$.

If $\overline L$ is a Hermitian line bundle
on $\Spec\mathcal O_K$ and if $s$ is an
arbitrary non-zero element in $L$, then
\[\widehat{\deg}_n(\overline L)=\frac{1}{[K:\mathbb Q]}
\Big(\log\#(L/\mathcal
O_Ks)-\sum_{\sigma:K\rightarrow\mathbb C
}\log\|s\|_\sigma\Big),\] which can also be
written as
\begin{equation}\label{Equ:degre d'Arakelove of a line bundle}
\widehat{\deg}_n(\overline L
)=-\frac{1}{[K:\mathbb
Q]}\Big(\sum_{\mathfrak
p}\log\|s\|_{\mathfrak
p}+\sum_{\sigma:K\rightarrow\mathbb
C}\log\|s\|_{\sigma}\Big).
\end{equation}
Note that this formula is analogous to the
degree function of a line bundle on a smooth
projective curve. Similarly to the geometric
case, for any Hermitian vector bundle
$\overline E$ of rank $r$ on $\Spec\mathcal
O_K$, we have
\begin{equation}\label{Equ:degre of determiant equals to degre}
\widehat{\deg}_n(\overline
E)=\widehat{\deg}_n(\Lambda^{r}\overline
E)\end{equation} where $\Lambda^r\overline E$
is the $r^{\text{th}}$ exterior power of
$\overline E$, that is, the {\it determinant}
of $\overline E$, which is a Hermitian line
bundle. Furthermore, if
$\xymatrix{0\ar[r]&\overline E'\ar[r]&
\overline E\ar[r]&\overline E''\ar[r]&0}$ is
a short exact sequence of Hermitian vector
bundles on $\Spec\mathcal O_K$, the following
equality holds:
\begin{equation}\label{Equ:degree et suite
exacte courte} \widehat{\deg}_n(\overline
E)=\widehat{\deg}_n( \overline
E')+\widehat{\deg}_n(\overline
E'').\end{equation}

\begin{lemma}
If $\overline E$ and $\overline F$ are two
Hermtian vector bundles of ranks $r_1$ and
$r_2$ on $\Spec\mathcal O_K$, respectively.
Then
\begin{equation}\label{Equ: degre d'Arakeovldu produt tensorei}
\widehat{\deg}_n(\overline E\otimes\overline
F)=\rang(E)\widehat{\deg}_n(\overline
F)+\rang(F)\widehat{\deg}_n(\overline E).
\end{equation}
\end{lemma}
\begin{proof}
The determinant Hermitian line bundle
$\Lambda^{r_1+r_2}(\overline
E\otimes\overline F)$ is isomorphic to
$(\Lambda^{r_1}\overline E)^{\otimes
r_2}\otimes(\Lambda^{r_2}\overline
F)^{\otimes r_1}$. Taking Arakelov degree and
using \eqref{Equ:degre of determiant equals
to degre} we obtain \eqref{Equ: degre
d'Arakeovldu produt tensorei}.
\end{proof}

We establish below the arithmetic analogue to
the first inequality in Proposition
\ref{Pro:encadrement de mu max otimes}.
\begin{proposition}
Let $\overline E$ and $\overline F$ be two
non-zero Hermitian vector bundles on
$\Spec\mathcal O_K$. Then
\[\widehat{\mu}_{\max}(\overline
E)+\widehat{\mu}_{\max}(\overline F)\le
\widehat{\mu}_{\max}(\overline
E\otimes\overline F).\]
\end{proposition}
\begin{proof}
Let $\overline E_{\des}$ and $\overline
F_{\des}$ be the Hermitian subbundles of
$\overline E$ and of $\overline F$
respectively as defined in Section
\ref{Sec:Introduction}. By definition,
$\widehat{\mu}(\overline
E_{\des})=\widehat{\mu}_{\max}(\overline E)$
and $\widehat{\mu}(\overline
F_{\des})=\widehat{\mu}_{\max}(\overline F)$.
Since $\overline E_{\des}\otimes\overline
F_{\des}$ is a Hermitian vector subbundle of
$\overline E\otimes \overline F$, we obtain
\[\widehat{\mu}_{\max}(\overline
E)+\widehat{\mu}_{\max}(\overline F)=
\widehat{\mu}(\overline E_{\des})+
\widehat{\mu}(\overline F_{\des})=
\widehat{\mu}(\overline
E_{\des}\otimes\overline
F_{\des})\le\widehat{\mu}_{\max}(\overline
E\otimes\overline F),\] where the second
equality results from \eqref{Equ: degre
d'Arakeovldu produt tensorei}.
\end{proof}

\begin{corollary}\label{Cor:minormation de maximal pent of tensor produc}
Let $(\overline E_i)_{1\le i\le n}$ be a
finite family of non-zero Hermitian vector
bundles on $\Spec\mathcal O_K$. Then the
following equality holds:
\begin{equation}
\widehat{\mu}_{\max}(\overline E_1)+ \cdots+
\widehat{\mu}_{\max}(\overline E_n)\le
\widehat{\mu}_{\max}(\overline
E_1\otimes\cdots\otimes\overline E_n).
\end{equation}
\end{corollary}

Let $\overline E$ and $\overline F$ be two
Hermitian vector bundles and
$\varphi:E_K\rightarrow F_K$ be a non-zero
$K$-linear homomorphism. For any finite place
$\mathfrak p$ of $K$, we denote by
$h_{\mathfrak p}(\varphi)$ the real number
$\log\|\varphi_{\mathfrak p}\|$, where
$\varphi_{\mathfrak p}:E_{K_{\mathfrak
p}}\rightarrow F_{K_{\mathfrak p}}$ is
induced from $\varphi$ by scalar extension.
Note that if $\varphi$ is induced by an
$\mathcal O_K$-homomorphism from $E$ to $F$,
then $h_{\mathfrak p}(\varphi)\le 0$ for any
finite place $\mathfrak p$. Similarly, for
any embedding $\sigma:K\rightarrow\mathbb C$,
we define
$h_{\sigma}(\varphi)=\log\|\varphi_\sigma\|$,
where $\varphi_{\sigma}:E_{\sigma,\mathbb
C}\rightarrow F_{\sigma,\mathbb C}$ is given
by the scalar extension $\sigma$. Finally, we
define the {\it height} of $\varphi$ as
\[h(\varphi)=\frac{1}{[K:\mathbb Q]}\Big(
\sum_{\mathfrak p}h_{\mathfrak p}(\varphi)+
\sum_{\sigma:K\rightarrow\mathbb
C}h_{\sigma}(\varphi)\Big).\]

\begin{proposition}[\cite{BostBour96}]
Let $\overline E$ and $\overline F$ be two
Hermitian vector bundles on $\Spec\mathcal
O_K$ and $\varphi:E_K\rightarrow F_K$ be a
$K$-linear homomorphism.
\begin{enumerate}[1)]
\item If $\varphi$ is injective, then \begin{equation}\label{Equ:slope
inequality}\widehat{\mu}(\overline
E)\le\widehat{\mu}_{\max}(\overline
F)+h(\varphi).\end{equation}
\item If $\varphi$ is non-zero, then
\begin{equation}\label{Equ:ineqality de pentes eux}
\widehat{\mu}_{\min}(\overline
E)\le\widehat{\mu}_{\max}(\overline
E)+h(\varphi)
\end{equation}
where $\widehat{\mu}_{\min}(\overline E)$ is
the minimal value of slopes of all non-zero
Hermitian vector quotient bundles of
$\overline E$.
\end{enumerate}
\end{proposition}

For any non-zero Hermitian vector bundle
$\overline E$ on $\Spec\mathcal O_K$, let
$\mathrm{u}\widehat{\deg}_n(\overline E)$ be
the maximal degree of line subbundles of
$\overline E$. We recall a result of Bost and
K\"unnemann comparing the maximal degree and
the maximal slope of $\overline E$, which is
a variant of Minkowski's First Theorem.

\begin{proposition}[\cite{Bost_Kunnemann} (3.27)] Let $\overline E$ be a
non-zero Hermitian vector bundle on
$\Spec\mathcal O_K$. Then
\begin{equation}\label{Equ:Bost-Kunnemann}
\mathrm{u}\widehat{\deg}_n(\overline E)\le
\widehat{\mu}_{\max}(\overline
E)\le\mathrm{u}\widehat{\deg}_n(\overline
E)+\frac{1}{2}\log(\rang
E)+\frac{\log|\Delta_K|}{2[K:\mathbb
Q]},\end{equation} where $\Delta_K$ is the
discriminant of $K$.
\end{proposition}

\section{Reminder on invariant theory}

\hskip\parindent In this section we recall
some known results in classical invariant
theory. We fix $K$ to be a field of
characteristic $0$. If $V$ is a vector space
over $K$ and if $u\in\mathbb N$, then the
expression $V^{\otimes(-u)}$ denotes the
space $V^{\vee\otimes u}$.

Let $V$ be a finite dimensional non-zero
vector space over $K$. For any $u\in\mathbb N
$, we denote by $J_u:\End_{K}(V)^{\otimes
u}\rightarrow\End_K(V^{\otimes u})$ the
$K$-linear homomorphism (of vector spaces)
which sends the tensor product
$T_1\otimes\cdots\otimes T_u$ of $u$ elements
in $\End_K(V)$ to their tensor product as an
endomorphism of $V^{\otimes u}$. The mapping
$J_u$ is actually a homomorphism of
$K$-algebras. Furthermore, as a homomorphism
of vector spaces, $J_u$ can be written as the
composition of the following natural
isomorphisms:
\[\xymatrix{\End_K(V)^{\otimes u}\ar[r]&(V^\vee\otimes
V)^{\otimes u} \ar[r]&(V^\vee)^{\otimes
u}\otimes V^{\otimes u}\ar[r]&(V^{\otimes
u})^\vee\otimes V^{\otimes
u}\ar[r]&\End_K(V^{\otimes u}),}\] so is
itself an isomorphism. Moreover, there exists
an action of the symmetric group
$\mathfrak{S}_{u}$ on $V^{\otimes u}$ by
permuting the factors. This representation of
$\mathfrak S_u$ defines a homomorphism from
the group algebra $K[\mathfrak S_u]$ to
$\End_K(V^{\otimes u})$. The elements of
$\mathfrak S_u$ act by conjugation on
$\End_K(V^{\otimes u})$. If we identify
$\End_K(V^{\otimes u})$ with
$\End_K(V)^{\otimes u}$ by the isomorphism
$J_u$, then the corresponding $\mathfrak
S_u$-action is just the permutation of
factors in tensor product. Finally the group
$\mathrm{GL}_K(V)$ acts diagonally on
$V^{\otimes u}$.

When $u=0$, $J_0$ reduces to the identical
homomorphism $\Id:K\rightarrow K$, and
$\mathfrak S_0$ reduces to the group of one
element. The ``diagonal'' action of
$\mathrm{GL}_K(V)$ on $V^{\otimes 0}\cong K$
is trivial.

We recall below the ``first principal
theorem'' of classical invariant theory (cf.
\cite{Weyl} Chapter III, see also
\cite{Atiyah_Bott_Patodi} Appendix 1 for a
proof).
\begin{theorem}\label{Thm:first pricipal theorem}
Let $V$ be a finite dimensional non-zero
vector space over $K$. Let $u\in\mathbb N$
and $v\in\mathbb Z$. If $T$ is a non-zero
element in $V^{\vee\otimes u}\otimes
V^{\otimes v}$, which is invariant by the
action of $\mathrm{GL}_K(V)$, then $u=v$, and
$T$ is a linear combination of permutations
in $\mathfrak S_u$ acting on $V$ (here we
identify $V^{\vee\otimes u}\otimes V^{\otimes
u}$ with $\End_K(V^{\otimes u})$).
\end{theorem}

We now present a generalization of Theorem
\ref{Thm:first pricipal theorem} to the case
of several linear spaces. In the rest of this
section, we fix a family $(V_i)_{1\le i\le
n}$ of finite dimensional non-zero vector
space over $K$. For any mapping
$\alpha:\{1,\cdots,n\}\rightarrow\mathbb Z$,
we shall use the notation
\begin{equation}\label{Equ:Valpha}V^\alpha:=V_1^{\otimes\alpha(1)}\otimes
\cdots\otimes
V_n^{\otimes\alpha(n)}\end{equation} to
simplify the writing.
 Denote
by $G$ the algebraic group
$\mathbb{GL}_K(V_1)\times_K\cdots\times_K\mathbb{GL}_K(V_n)$.
Then $G(K)$ is the group
$\mathrm{GL}_K(V_1)\times\cdots\times\mathrm{GL}_K(V_n)$.
For any mapping
$\alpha:\{1,\cdots,n\}\rightarrow\mathbb N$
with natural integer values, we denote by
$\mathfrak{S}_{\alpha}$ the product
$\mathfrak{S}_{\alpha(1)}\times
\cdots\times\mathfrak{S}_{\alpha(n)}$ of
symmetric groups. We have a natural
isomorphism of $K$-algebras from
$\End_K(V^\alpha)$ to $\End_K(V_1)^{\otimes
\alpha(1)}\otimes_K\cdots\otimes_K
\End_K(V_n)^{\otimes \alpha(n)}$. The group
$G(K)$ acts naturally on $V^\alpha$ and the
group $\mathfrak S_{\alpha}$ acts on
$V^\alpha$ by permutating tensor factors. By
using induction on $n$, Theorem
\ref{Thm:first pricipal theorem} implies the
following corollary:

\begin{corollary}\label{Thm:premier theoreme principal de la theorie d'invaraint}
With the notation above, if
$\alpha:\{1,\cdots,n\}\rightarrow \mathbb N$
and $\beta:\{1,\cdots,n\}\rightarrow\mathbb
Z$ are two mappings and if $T$ is a non-zero
element in $(V^\alpha)^\vee\otimes V^\beta $
which is invariant by the action of $G(K)$,
then $\alpha=\beta$, and $T$ is a linear
combination of elements in $\mathfrak
S_{\alpha}$ acting on $V^\alpha$.
\end{corollary}

Let $\mathcal A$ be a finite family of
mappings from $\{1,\cdots,n\}$ to $\mathbb N$
and $(b_i)_{1\le i\le n}$ be a family of
integers. We denote by $W$ the vector space
$\bigoplus_{\alpha\in\mathcal A} V^\alpha $.
Note that the group $G(K)$ acts naturally on
$W$. Let $L$ be the $G(K)$-module $(\det
V_1)^{\otimes b_1}\otimes\cdots\otimes(\det
V_n)^{\otimes b_n}$. For any integer $D\ge 1$
and any element
$\underline{\alpha}=(\alpha_j)_{1\le j\le
D}\in\mathcal A^D$, let
\[\pr_{\underline{\alpha}}:W^{\otimes D}\longrightarrow
V^{\alpha_1}\otimes\cdots\otimes
V^{\alpha_D}\] be the canonical projection.
For any integer $i\in\{1,\cdots,n\}$, let
$r_i$ be the rank of $V_i$ over $K$. Finally
let $\pi:\mathbb P(W^\vee)\rightarrow\Spec K$
be the canonical morphism.

\begin{theorem}\label{Thm:semistability et determinant}
With the notation above, if $m$ is a strictly
positive integer and if $R$ is a vector
subspace of rank $1$ of $W$ (considered as a
rational point of $\mathbb P(W^\vee)$) which
is semistable for the action of ${G}$
relatively to $\mathcal
O_{W^\vee}(m)\otimes\pi^*L$, then there
exists an integer $D\ge 1$ and a family
$\underline{\alpha}=(\alpha_j)_{1\le j\le
mD}$ of elements in $\mathcal A$ such that,
by noting $A=\alpha_1+\cdots+\alpha_{mD}$, we
have $A(i)=Db_ir_i$ and hence $b_i\ge 0$ for
any $i$.

Furthermore, there exists an element
$\sigma\in\mathfrak S_{A}$ such that the
composition of homomorphisms
\[\xymatrix{\relax R^{\otimes mD}\otimes
L^{\vee\otimes D} \ar[r]&W^{\otimes
mD}\otimes L^{\vee\otimes D
}\ar[rr]^-{\pr_{\underline{\alpha}}\otimes\Id}&&
V^A\otimes L^{\vee\otimes
D}\ar[d]_{\sigma\otimes\Id}\\
&&&V^A\otimes L^{\vee\otimes
D}\ar[d]_{{\det_{\hskip -1 pt\raisebox{-1.2
pt }{$\scriptscriptstyle V_1$}}}^{\hskip -8.5
pt\otimes \hskip -0.5pt
Db_1}\otimes\cdots\otimes{\det_{\hskip -1
pt\raisebox{-1.2 pt }{$\scriptscriptstyle
V_n$}}}^{\hskip -9 pt\otimes\hskip -0.5pt Db_n}\otimes\Id}\\
&&&L^{\otimes D}\otimes L^{\vee\otimes D}\cong K }\]
does not vanish, where the first arrow is induced by
the canonical inclusion of $R^{\otimes nD}$ in
$W^{\otimes nD}$.
\end{theorem}
\begin{proof}
Since $R$ is semistable for the action of
${G}$ relatively to $\mathcal
O_{W^\vee}(m)\otimes\pi^*L$, there exists an
integer $D\ge 1$ and an element $s\in
S^{mD}(W^\vee)\otimes L^{\otimes D}$ which is
invariant by the action of $G(K)$ such that
the composition of homomorphisms
\[\xymatrix{\relax R^{\otimes mD}\otimes L^{\vee\otimes D}
\ar[r]& S^{mD}(W^\vee)^\vee\otimes
L^{\vee\otimes D}\ar[r]^-s&K}\] does not
vanish, the first arrow being the canonical
inclusion.

As $K$ is of characteristic $0$,
$S^{md}(W^\vee)$ is a direct factor as a
$\mathrm{GL}(W)$-module of $W^{\vee\otimes
mD}$. Hence $S^{mD}(W^\vee)\otimes L^{\otimes
D}$ is a direct factor as a $G(K)$-module of
$W^{\vee\otimes mD}\otimes L^{\otimes D}$. So
we can choose $s'\in W^{\vee\otimes
mD}\otimes L^{\otimes D}$ invariant by the
action of $G(K)$ such that the class of $s'$
in $S^{mD}(W^\vee)\otimes L^{\otimes D}$
coincides with $s$. There then exists
$\underline{\alpha}=(\alpha_j)_{1\le j\le mD
}\in\mathcal A^D$ such that the composition
\[\xymatrix{\relax R^{\otimes mD}\otimes L^{\vee\otimes D}
\ar[r]&W^{\otimes mD}\otimes L^{\vee\otimes
D}\ar[rr]^-{\pr_{\underline{\alpha}}\otimes\Id}&&
V^A\otimes L^{\vee\otimes
D}\ar[r]^-{s_{\underline{\alpha}}'}&K}\] is
non-zero, where
$A=\alpha_1+\cdots+\alpha_{mD}$ and
$s_{\underline{\alpha}}'$ is the component of
index $\underline{\alpha}$ of $s'$. Let
$B:\{1,\cdots,n\}\rightarrow\mathbb Z$ be the
mapping which sends $i$ to $Db_ir_i$. Note
that for any $i$, $\Lambda^{r_i}V_i=\det V_i$
is naturally a direct factor of $V_i^{\otimes
r_i}$. We can therefore choose a preimage
$s_{\underline{\alpha}}''$ of
$s_{\underline{\alpha}}'$ in
$(V^A)^{\vee}\otimes {V^B}$ which is
invariant by $G(K)$. By Corollary
\ref{Thm:premier theoreme principal de la
theorie d'invaraint}, $A=B$ and
$s_{\underline{\alpha}}''$ is a linear
combination of permutations acting on $V$.
Therefore the theorem is proved.
\end{proof}

\section{Upper bound for the degree of a Hermitian line subbundle
with hypothesis of semistability
}\label{Sec:Upper bound with hypothesis of
semi-stability }

\hskip\parindent Let $K$ be a number field
and $\mathcal O_K$ be its integer ring.
Consider a family $(\overline E_i)_{1\le i\le
n}$ of non-zero Hermitian vector bundles on
$\Spec\mathcal O_K$. Let $\mathcal A$ be a
non-empty and finite family of
non-identically zero mappings from
$\{1,\cdots,n\}$ to $\mathbb N$. We define a
new Hermitian vector bundle over
$\Spec\mathcal O_K$ as follows:
\[\overline E:=\bigoplus_{{\alpha}\in\mathcal A}\overline E_1^{\otimes\alpha(1)}
\otimes\cdots\otimes\overline
E_n^{\otimes\alpha(n)}.\]

In this section, we shall use the ideas in
\cite{Bost94} to obtain an upper bound for
the Arakelov degree of a Hermitian line
subbundle $\overline M$ of $\overline E$
under hypothesis of semistability (in the
sense of geometric invariant theory) for
$M_K$. This upper bound is crucial because,
as we shall see later, the general case can
be reduced to this special one through an
argument of Ramanan and Ramanathan
\cite{Ramanan_Ramanathan}.

For any integer $i$ such that $1\le i\le n$, let $r_i$ be
the rank of $E_i$ and let $V_i$ be the vector
space $E_{i,K}$. Let $W=E_{K}$ and
$\pi:\mathbb P(W^\vee)\rightarrow\Spec K$ be
the canonical morphism. By definition
$W=\bigoplus_{\underline{\alpha}\in\mathcal
A}V^\alpha$, where $V^\alpha$ is defined in
\eqref{Equ:Valpha}. We denote by ${G}$ the
algebraic group
$\mathbb{GL}_K(V_1)\times\cdots\times\mathbb{GL}_K(V_n)$
which acts naturally on $\mathbb P(W^\vee)$.
Let $(b_i)_{1\le i\le n}$ be a family of
strictly positive integers such that $r_i$
divides $b_i$. Finally let \[\overline
L=(\Lambda^{r_1}\overline E_1)^{\otimes
b_1/r_1}\otimes\cdots(\Lambda^{r_n}\overline
E_n)^{\otimes b_n/r_n}.\]

\begin{lemma}\label{Lem:hauteur de det}
Let $H$ be a Hermitian space of dimension
$d>0$. Then the norm of the homomorphism
$\det:H^{\otimes d}\rightarrow\Lambda^dH$
equals $\sqrt{d!}$.
\end{lemma}
\begin{proof} Let $(e_i)_{1\le i\le d}$ be an
orthonormal basis of $H$ and let
$(e_i^\vee)_{1\le i\le d}$ be its dual basis
in $H^\vee$. If we identifies $\Lambda^dH$
with $\mathbb C$ via the basis
$e_1\wedge\cdots\wedge e_d$, then the
homomorphism $\det$, viewed as an element in
$H^{\vee\otimes d}$, can be written as
\[\sum_{\sigma\in\mathfrak S_d}\mathrm{sign}(\sigma)
e_{\sigma(1)}\otimes\cdots\otimes
e_{\sigma(d)},
\]
which is the sum of $d!$ orthogonal vectors
of norm $1$ in $H^{\vee\otimes d}$. So its
norm is $\sqrt{d!}$.
\end{proof}

\begin{theorem}\label{Thm:majoration de degre d'Arakelov avec
l'hypo de semistability} With the notation
above, if $m\ge 1$ is an integer and if
$\overline M$ is a Hermitian line subbundle
of $\overline E$ such that $M_K$ is
semistable for the action of ${G}$ relatively
to $\mathcal O_{W^\vee}(m)\otimes\pi^*L_K$,
then
\begin{equation*}\widehat{\deg}(\overline
M)\le\frac{1}{m}\widehat{\deg}(\overline L
)+\frac{1}{2m}\sum_{i=1}^rb_i\log(\rang E_i
)=\sum_{i=1}^n\frac{b_i}{m}
\Big(\widehat{\mu}(\overline E_i)+\frac{1}{2}\log(\rang
E_i)\Big).\end{equation*}
\end{theorem}
\begin{proof}
By Theorem \ref{Thm:semistability et
determinant}, we get, by combining the slope
inequality \eqref{Equ:slope inequality} and
Lemma \ref{Lem:hauteur de det},
\[\begin{split}&\quad\;mD\widehat{\deg}(\overline M)-D\widehat{\deg}(\overline L)=
mD\widehat{\deg}(\overline
M)-\sum_{i=1}^nDb_i\widehat{\mu}(\overline
E_i)\\
&\le\sum_{i=1}^n\frac{A(i)\log(r_i!)}{2r_i}=\sum_{i=1}^n
\frac{Db_i\log(r_i!)}{2r_i} \le\frac 12 D\sum_{i=1}^n
b_i\log r_i,
\end{split}
\]
where we have used the evident estimation
$r!\le r^r$ to obtain the last inequality.
Finally we divide the inequality by $mD$ and
obtain
\[\widehat{\deg}({\overline M})\le\frac{1}{m}\widehat{\deg}
(\overline L)+\frac{1}{2m}\sum_{i=1}^nb_i\log r_i=
\sum_{i=1}^n\frac{b_i}{m}\Big(\widehat{\mu}(\overline
E_i)+\frac{\log r_i}{2}\Big).\]
\end{proof}

Let $m$ be a strictly positive integer which
is divisible by all $r_i$. We apply Theorem
\ref{Thm:majoration de degre d'Arakelov avec
l'hypo de semistability} to the special case
where $\mathcal A$ contains a single map
$\alpha$ such that $\alpha(i)=1$ for any
$i\in\{1,\cdots,n\}$, in other words,
$\overline E=\overline
E_1\otimes\cdots\otimes\overline E_n$, and
where $b_i=m$ for any integer $i$ such that $1\le i\le n$. Then we
get the following upper bound:
\begin{corollary}
If $\overline M$ is a Hermitian line
subbundle of $\overline
E_1\otimes\cdots\otimes\overline E_n$ such
that $M_K$ is semistable for the action of
$G$ relatively to $\mathcal
O_{W^\vee}(m)\otimes\pi^*L_K$, then we have
\begin{equation}\label{Equ:subline bundle semistable}\widehat{\deg}(\overline
M)\le\sum_{i=1}^n
\Big(\widehat{\mu}(\overline E_i)+\frac
12\log(\rang E_i )\Big).\end{equation}
\end{corollary}

\section{Filtrations of vector spaces}
\label{Sec:filtrations}

\hskip\parindent In this section, we
introduce some basic notation and results on
$\mathbb R$-filtrations of vector spaces,
which we shall use in the sequel. We fix a
field $K$.

\subsection{Definition of filtrations}
\hskip\parindent Let $V$ be a non-zero vector
space of finite rank $r$ over $K$. We call
$\mathbb R$-{\it filtration} of $V$ any
family $\mathcal F=(\mathcal F_\lambda
V)_{\lambda\in\mathbb R }$ of subspaces of
$V$ such that
\begin{enumerate}[1)]
\item $\mathcal F_\lambda V\supset\mathcal
F_{\lambda'}V$ for all $\lambda\le\lambda'$,
\item $\mathcal F_\lambda V=0$ for $\lambda$ sufficiently
positive,
\item $\mathcal F_\lambda V=V$ for
$\lambda$ sufficiently negative,
\item the
function $x\mapsto\rang_K(\mathcal F_xV)$ on
$\mathbb R$ is left continuous.
\end{enumerate}
A filtration $\mathcal F$ of $V$ is
equivalent to the data of a flag
\begin{equation}\label{Equ:drapoeau correpondant a
F}V=V_0\supsetneq V_1\supsetneq
V_2\supsetneq\cdots\supsetneq
V_d=0\end{equation} of $V$ together with a
strictly increasing sequence of real numbers
$(\lambda_i)_{0\le i<d}$. In fact, we have
the relation $\mathcal F_\lambda V=
\bigcup_{\lambda_i\ge\lambda}V_i$. We define
the {\it expectation} of $\mathcal F$ to be
\begin{equation}\label{Equ:expectation of a filtration}
\mathbb E[\mathcal F]:=\sum_{i=0}^{d-1}
\frac{\rang_K(V_{i}/V_{i+1})}
{\rang_KV}\lambda_i.
\end{equation}
Furthermore, we define a function
$\lambda_{\mathcal F}:V\rightarrow\mathbb
R\cup\{+\infty\}$ such that
\begin{equation}\label{Equ:lambda F(x)}\lambda_{\mathcal
F}(x)=\sup\{a\in\mathbb R\,|\,x\in\mathcal
F_aV\}.\end{equation} The function
$\lambda_{\mathcal F}$ takes values in
$\{\lambda_0,\cdots,\lambda_{d-1}\}\cup\{+\infty\}$
and is finite on $\mathbb R\setminus\{0\}$.

\subsection{Spaces of filtrations}

\hskip\parindent Let $Z$ be a subset of
$\mathbb R$. We say that $\mathcal F$ is {\it
supported by $Z$} if $\{\lambda_i\;|\;0\le
i<d\}\subset Z$. We say that a basis
$\mathbf{e}$ of $V$ is {\it compatible} with
$\mathcal F$ if it is compatible with the
flag \eqref{Equ:drapoeau correpondant a F}.
That is, $\#(V_i\cap\mathbf{e})=\rang(V_i)$.

We denote by $\mathbf{Fil}_V$ the set of all
filtrations of $V$. For any non-empty subset
$Z$ of $\mathbb R$, denote by
$\mathbf{Fil}_V^Z$ the set of all filtrations
of $V$ supported by $Z$. Finally, for any
basis $\mathbf{e}$, we use the expression
$\mathbf{Fil}_{\mathbf{e}}$ to denote the set
of all filtrations of $V$ with which
$\mathbf{e}$ is compatible, and we denote by
$\mathbf{Fil}_{\mathbf{e}}^Z$ the subset of
$\mathbf{Fil}_{\mathbf{e}}$ of filtrations
supported by $Z$.

\begin{proposition}
Let $\mathbf{e}=(e_1,\cdots,e_r)$ be a basis
of $V$ and $Z$ be a non-empty subset of
$\mathbb R$. The mapping
$\Phi_{\mathbf{e}}:\mathbf{Fil}_{\mathbf{e}}^Z\rightarrow
Z^r$ defined by
\begin{equation}\label{Equ:bijection Phie(mathcal F)}
\Phi_{\mathbf{e}}(\mathcal
F)=(\lambda_{\mathcal
F}(e_1),\cdots,\lambda_{\mathcal F}(e_r))
\end{equation}
is a bijection.
\end{proposition}

\begin{proposition}\label{Pro:lambda mathcal F as
inf des formel ineares} Let $v$ be a non-zero
vector in $V$, $F$ be a subfield of $\mathbb
R$ and $\mathbf{e}$ be a basis of $V$. Then
the function $\mathcal
F\mapsto\lambda_{\mathcal F }(v)$ from
$\mathbf{Fil}_{\mathbf{e}}^F$ to $\mathbb R$
can be written as the minimal value of a
finite number of $F$-linear forms.
\end{proposition}
\begin{proof}
Let $v=\sum_{i=1}^r a_ie_i$ be the
decomposition of $v$ in the basis
$\mathbf{e}$, then for any filtration
$\mathcal F\in\mathbf{Fil}_{\mathbf{e}}^F$,
we have
\[\lambda_{\mathcal F}(v)=\min_{\begin{subarray}{c}
1\le i\le n\\ a_i\neq 0
\end{subarray}} \lambda_{\mathcal F}(e_i).\]
\end{proof}

\subsection{Construction of filtrations}

\hskip\parindent For any real number
$\varepsilon>0$, we define the {\it dilation}
of $\mathcal F$ by $\varepsilon$ as the
filtration
\begin{equation}\label{Equ:dilation definition}\psi_{\varepsilon}\mathcal
F:=(\mathcal
F_{\varepsilon\lambda}V)_{\lambda\in\mathbb
R}\end{equation} of $V$. Clearly we have
\begin{equation}\label{Equ:dilation}\mathbb E[\psi_{\varepsilon}\mathcal F]
=\varepsilon\mathbb E[\mathcal
F]\qquad\text{and}\qquad\lambda_{\psi_{\varepsilon}\mathcal
F}=\varepsilon\lambda_{\mathcal F}.\end{equation}

Let $(V^{(i)})_{1\le i\le n}$ be a family of
non-zero vector spaces of finite rank over
$K$ and $V=\bigoplus_{i=1}^n V^{(i)}$ be
their direct sum. For each integer $1\le i\le
n$, let $\mathcal F^{(i)}$ be a filtration of
$V^{(i)}$. We construct a filtration
$\mathcal F$ of $V$ such that
\[\mathcal F_\lambda V=\bigoplus_{i=1}^n\mathcal F_\lambda^{(i)} V^{(i)}.\]
The filtration $\mathcal F$ is called the
{\it direct sum} of $\mathcal F^{(i)}$ and is
denoted by $\mathcal
F^{(1)}\oplus\cdots\oplus\mathcal F^{(n)}$.
If for each $1\le i\le n$, $\mathbf{e}^{(i)}$
is a basis of $V^{(i)}$ which is compatible
with $\mathcal F^{(i)}$, then the disjoint
union
$\mathbf{e}^{(1)}\amalg\cdots\amalg\mathbf{e}^{(n)}$,
which is a basis of
$V^{(1)}\oplus\cdots\oplus V^{(n)}$, is
compatible with $\mathcal
F^{(1)}\oplus\cdots\oplus\mathcal F^{(n)}$.
Similarly, if $W=\bigotimes_{i=1}^n V^{(i)}$
is the tensor product of $V^{(i)}$, we
construct a filtration $\mathcal G$ of $W$
such that
\[\mathcal G_\lambda W=\sum_{\lambda_1+\cdots+\lambda_n
\ge\lambda}\bigotimes_{i=1}^n\mathcal
F^{(i)}_{\lambda_i}V^{(i)},\] called the {\it
tensor product} of $\mathcal F^{(i)}$ and
denoted by $\mathcal
F^{(1)}\otimes\cdots\otimes\mathcal F^{(n)}$.
If $\mathbf{e}^{(i)}$ is a basis
of $V^{(i)}$ which is compatible with the
filtration $\mathcal F^{(i)}$, then the basis
\[\mathbf{e}^{(1)}\otimes\cdots\otimes\mathbf{e}^{(n)}:=
\{e_1\otimes\cdots\otimes e_n\,|\,\forall
1\le i\le n,\; e_i\in\mathbf{e}^{(i)}\}\] of
$V^{(1)}\otimes\cdots\otimes V^{(n)}$ is
compatible with $\mathcal
F^{(1)}\otimes\cdots\otimes\mathcal F^{(n)}$.
Finally, for any $\varepsilon>0$,
\begin{equation}\label{Equ:dilatation et produit tensoriel}
\psi_{\varepsilon}(\mathcal
F^{(1)}\otimes\cdots\otimes \mathcal
F^{(n)})=\psi_{\varepsilon}\mathcal F^{(1)}\otimes
\cdots\otimes\psi_{\varepsilon}\mathcal F^{(n)}.
\end{equation}

\subsection{Scalar product on the space of filtrations}

\hskip\parindent Let $V$ be a non-zero vector
space of finite rank $r$ over $K$. If
$\mathcal F$ and $\mathcal G$ are two
filtrations of $V$, then by Bruhat's
decomposition, there always exists a basis
$\mathbf{e}$ of $V$ which is compatible
simultaneously with $\mathcal F$ and
$\mathcal G$. We define the {\it scalar
product} of $\mathcal F$ and $\mathcal G$ as
\begin{equation}\langle\mathcal F,\mathcal G\rangle:=
\frac{1}{r}\sum_{i=1}^r\lambda_{\mathcal F
}(e_i)\lambda_{\mathcal
G}(e_i).\end{equation} This definition does
not depend on the choice of $\mathbf{e}$. The
number $\|\mathcal F\|:=\langle\mathcal
F,\mathcal F\rangle^{\frac 12}$ is called the
{\it norm} of the filtration $\mathcal F$.
Notice that $\|\mathcal F\|=0$ if and only if
$\mathcal F$ is supported by $\{0\}$. In this
case, we say that the filtration $\mathcal F$
is {\it trivial}.
\begin{proposition}\label{Pro:euclidean space of filtrations}
Let $\mathbf{e}$ be a basis of $V$. Then the
function $(x,y)\mapsto r
\langle\Phi_{\mathbf{e}}^{-1}(x),
\Phi_{\mathbf{e}}^{-1}(y)\rangle$ on $\mathbb
R^r\times\mathbb R^r$ coincides with the
usual Euclidean product on $\mathbb R^r$,
where
$\Phi_{\mathbf{e}}:\mathbf{Fil}_{\mathbf{e}}\rightarrow
\mathbb R^r$ is the bijection defined in
\eqref{Equ:bijection Phie(mathcal F)}.
\end{proposition}

\subsection{Construction of filtration from subquotients}
\label{Subsec:Construction of filtration from
subquotients} \hskip\parindent Let $V$ be a
non-zero vector space of finite rank over $K$
and $\mathcal F$ be a filtration of $V$
corresponding to the flag $V=V_0\supsetneq
V_1\supsetneq V_2\supsetneq\cdots \supsetneq
V_d=0$ together with the sequence
$(\lambda_j)_{0\le j< d}$. For any integer $j$ such that
$0\le j< d$, we pick a basis $\mathbf{e}^j$
of the subquotient $V_{j}/V_{j+1}$. After
choosing a preimage of $\mathbf{e}^j$ in
$V_{j}$ and taking the disjoint union of the
preimages, we get a basis
$\mathbf{e}=(e_1,\cdots,e_r)$ of $V$ which is
clearly compatible with the filtration
$\mathcal F$. The basis $\mathbf{e}$ defines
a natural isomorphism $\Psi$ form $V$ to
$\bigoplus_{j=0}^{d-1}(V_{j}/V_{j+1})$ which
sends $e_i$ to its class in
$V_{\tau(i)}/V_{\tau(i)+1}$, where
$\tau(i)=\max\{j\,|\,e_i\in V_j\}$.

For any integer $j$ such that $0\le j\le d-1$, let
$\mathcal G^j$ be a filtration of
$V_{j}/V_{j+1}$ with which $\mathbf{e}^j$ is
compatible. We construct a filtration
$\mathcal G$ on $V$ which is the direct sum
via $\Psi$ of $(\mathcal G^j)_{0\le j\le
d-1}$. Note that the basis $\mathbf{e}$ is
compatible with the new filtration $\mathcal
G$. If $e_i$ is an element in $\mathbf{e}$,
then $\lambda_{\mathcal
G}(e_i)=\lambda_{\mathcal
G^{\tau(i)}}(\Psi(e_i))$. Therefore we have
\begin{equation}\label{Equ:esperance de filtration sousquotient}
\mathbb E[\mathcal G]=\frac{1}{r}\sum_{j=0}^{d-1}
\mathbb E[\mathcal G^j
]{\rang_K(V_j/V_{j+1})},\qquad\langle\mathcal
F,\mathcal G\rangle=\frac{1}{r}\sum_{j=0}^{d-1}
\lambda_j\mathbb E[\mathcal
G^j]\rang_K(V_j/V_{j+1}).\end{equation}

\section{More facts in geometric invariant theory}
\label{Sec:More facts in geometric invariant
theory} \hskip\parindent We shall establish
in this section the explicit version of a
result of Ramanan and Ramanathan
\cite{Ramanan_Ramanathan} (Proposition 1.12)
for our particular purpose, along the path
indicated by Totaro \cite{Totaro96} in his
proof of Fontaine's conjecture.

Let $K$ be a perfect field. If $G$ is a
reductive group over $\Spec K$, we call {\it
one-parameter subgroup} of $G$ any morphism
of $K$-group schemes from $\mathbb
G_{\mathrm{m},K}$ to $G$. Let $X$ be a
$K$-scheme on which $G$ acts. If $x$ is a
rational point of $X$ and if $h$ is a
one-parameter subgroup of $G$, then we get a
$K$-morphism from $\mathbb G_{\mathrm{m},K}$
to $X$ given by the composition
\[\xymatrix{\relax \mathbb G_{\mathrm{m},K}\ar[r]^-h&G\ar[r]^-{\sim}& G\times_K\Spec
K\ar[r]^-{\Id\times
x}&G\times_KX\ar[r]^-{\sigma}&X},\] where
$\sigma$ is the action of the group. If in
addition $X$ is proper over $\Spec K$, this
morphism extends in the unique way to a
$K$-morphism $f_{h,x}$ from $\mathbb A_K^1$
to $X$. We denote by $0$ the unique element
in $\mathbb A^1(K)\setminus\mathbb
G_{\mathrm{m}}(K)$. The morphism $f_{h,x}$
sends the point $0$ to a rational point of
$X$ which is invariant by the action of
$\mathbb G_{\mathrm{m},K}$. If $L$ is a
$G$-linearized line bundle on $X$, then the
action of $\mathbb G_{\mathrm{m},K}$ on
$L|_{f_{h,x}(0)}$ defines a character of
$\mathbb G_{\mathrm{m},K}$ of the form
\[t\mapsto t^{\mu(x,h,L)},\text{ where }\mu(x,h,L)\in\mathbb Z.\] Furthermore, if we
denote by $\Pic^G(X)$ the group of
isomorphism classes of all $G$-linearized
line bundles, then $\mu(x,h,\cdot)$ is a
homomorphism of groups from $\Pic^G(X)$ to
$\mathbb Z$.
\begin{remark}
In \cite{Mumford94}, the authors have defined
the $\mu$-invariant with a minus sign.
\end{remark}

We now recall a well-known result which gives
a semistability criterion for rational points
in a projective variety equipped with an
action of a reductive group.
\begin{theorem}[Hilbert-Mumford-Kempf-Rousseau]
\label{Thm:Hilbert-Mumford criterion} Let $G$
be a reductive group which acts on a
projective variety $X$ over $\Spec K$, $L$ be
an ample $G$-linearized line bundle on $X$
and $x\in X(K)$ be a rational point. The
point $x$ is semistable for the action of $G$
relatively to $L$ if and only if
$\mu(x,h,L)\ge 0$ for any one-parameter
subgroup  $h$ of $G$.
\end{theorem}

This theorem has been originally proved by
Mumford (see \cite{Mumford94}) for the case
where $K$ is algebraically closed. Then it
has been independently proved in all
generality by Kempf \cite{Kempf78} and
Rousseau \cite{Rousseau78}, where Kempf's
approach has been revisited by Ramanan and
Ramanathan \cite{Ramanan_Ramanathan} to prove
that the tensor product of two semistable
vector bundle on a smooth curve (over a
perfect field) is also semistable. The idea
of Kempf is to choose a special one-parameter
subgroup $h_0$ of $G$ destabilizing $x$,
which minimizes a certain function. The
uniqueness of his construction allows us to
descend to a smaller field. Later Totaro
\cite{Totaro96} has introduced a new approach
of Kempf's construction and thus found an
elegant proof of Fontaine's conjecture.

In the rest of this section, we recall
Totaro's approach of Hilbert-Mumford
criterion in our setting. We begin by
calculating explicitly the number
$\mu(x,h,L)$ using filtrations introduced in
the previous section.

Let $V$ be a vector space of finite rank over
$K$ and $\rho:G\rightarrow\mathbb{GL}(V)$ be
a representation of $G$ on $V$. If $h:\mathbb
G_{\mathrm{m},K}\rightarrow G$ is a
one-parameter subgroup, then the
multiplicative group $\mathbb
G_{\mathrm{m},K}$ acts on $V$ via $h$ and
$\rho$. Hence we can decompose $V$ into
direct sum of eigenspaces. More precisely, we
have the decomposition
$V=\bigoplus_{i\in\mathbb Z}V(i)$, where the
action of $\mathbb G_{\mathrm{m},K}$ on
$V(i)$ is given by the composition
\[\xymatrix{\relax \mathbb G_{\mathrm{m},K}\times_K V(i)\ar[rr]
^-{(t\mapsto t^i)\times\Id}&&\mathbb
G_{\mathrm{m},K}\times_KV(i)\ar[r]&V(i)},\]
the second arrow being the scalar
multiplication structure on $V(i)$. We then
define a filtration $\mathcal F^{\rho,h}$
(supported by $\mathbb Z$) of $V$ such that
\[\mathcal F_{\lambda}^{\rho,h}V=\sum_{i\ge\lambda}V(i)\qquad
\qquad\text{where }\lambda\in\mathbb R,\]
called the {\it filtration associated to $h$}
relatively to the representation $\rho$. If
there is no ambiguity on the representation,
we also write $\mathcal F^h$ instead of
$\mathcal F^{\rho,h}$ to simplify the
notation. If $G=\mathbb{GL}(V)$ and if $\rho$
is the canonical representation, then for any
filtration $\mathcal F$ of $V$ supported by
$\mathbb Z$, there exists a one-parameter
subgroup $h$ of $G$ such that the filtration
associated to $h$ equals $\mathcal F$.

From the scheme-theoretical point of view,
the algebraic group $G$ acts via the
representation $\rho$ on the projective space
$\mathbb P(V^\vee)$.

The following result is in
\cite{Mumford94}
Proposition 2.3. Here we work on the dual space $V^\vee$.

\begin{proposition}\label{Pro:Totaro 1}
Let $x$ be a rational point of $\mathbb
P(V^\vee)$, viewed as a one-dimensional
subspace of $V$ and let $v_x$ be an arbitrary
non-zero vector in $x$. Then
\[\mu(x,h,\mathcal O_{V^\vee}(1))=
-\lambda_{\mathcal F^{\rho,h}}(v_x),\] where
the function $\lambda_{\mathcal F^{\rho,h}}$
is defined in \eqref{Equ:lambda F(x)}.
\end{proposition}
\begin{proof}
Let $v_x=\sum_{i\in\mathbb Z}v_x(i)$ be the
canonical decomposition of $v_x$. Let
$i_0=\lambda_{\mathcal F^{\rho,h}}(v_x)$. By
definition, it is the maximal index $i$ such
that $v_x(i)$ is non-zero. Furthermore,
$f_{h,x}(0)$ is just the rational point $x_0$
which corresponds to the subspace of $V$
generated by $v_x(i_0)$. The restriction of
$\mathcal O_{V^\vee}(1)$ on $x_0$ identifies
with the quotient $(Kv_x(i_0))^\vee$ of
$V^\vee$. Since the action of $\mathbb
G_{\mathrm{m},K}$ on $v_x(i_0)$ via $h$ is
the multiplication by $t^{i_0}$, its action
on $(Kv_x(i_0))^\vee$ is then the
multiplication by $t^{-i_0}$. Therefore,
$\mu(x,h,\mathcal
O_{V^\vee}(1))=-i_0=-\lambda_{\mathcal
F^{\rho,h}}(v_x)$.
\end{proof}

Let $(V_i)_{1\le i\le n}$ be a finite family
of non-zero vector spaces of finite rank over
$K$. For any integer $1\le i\le n$, let $r_i$
be the rank of $V_i$. Let $G$ be the
algebraic group
$\mathbb{GL}(V_1)\times\cdots\times\mathbb{GL}(V_n)$.
We suppose that the algebraic group $G$ acts
on a vector space $V$. Let $\pi:\mathbb
P(V^\vee)\rightarrow\Spec K$ be the canonical
morphism. For each integer $1\le i\le n$, we
choose an integer $m_i$ which is divisible by
$r_i$. Let $M$ be the $G$-linearized line
bundle on $\mathbb P(V^\vee)$ defined as
\[M:=\bigotimes_{i=1}^n
\pi^*(\Lambda^{r_i} V_i)^{\otimes m_i/r_i}.\]
It is a trivial line bundle on $\mathbb
P(V^\vee)$ with possibly non-trivial
$G$-action. Notice that any one-parameter
subgroup of $G$ is of the form
$h=(h_1,\cdots,h_n)$, where $h_i$ is a
one-parameter subgroup of $\mathbb{GL}(V_i)$.
Let $\mathcal F^{h_i}$ be the filtration of
$V_i$ associated to $h_i$ relatively to the
canonical representation of
$\mathbb{GL}(V_i)$ on $V_i$. The action of
$\mathbb G_{\mathrm{m},K}$ via $h_i$ on
$\Lambda^{r_i} V_i$ is nothing but the
multiplication by $t^{r_i\mathbb E[\mathcal
F^{ h_i}]}$. Then we get the following
result.
\begin{proposition}\label{Totaro 2}
With the notation above, for any rational
point $x$ of $\mathbb P(V^\vee)$, we have
\[\mu(x,h,M)=\sum_{i=1}^n m_i\mathbb E[\mathcal F^{h_i}].\]
\end{proposition}

We now introduce the Kempf's destabilizing
flag for the action of a finite product of
general linear groups. Consider a family
$(V^{(i)})_{1\le i\le n}$ of finite
dimensional non-zero vector space over $K$.
Let $W$ be the tensor product
$V^{(1)}\otimes_K\cdots\otimes_K V^{(n)}$ and
$G$ be the algebraic group
$\mathbb{GL}(V^{(1)})\times\cdots\times\mathbb{GL}(V^{(n)})$.
For any integer $i$ such that $1\le i\le n$, let $r^{(i)}$
be the rank of $V^{(i)}$. The group $G$ acts
naturally on $W$ and hence on $\mathbb
P(W^\vee)$. We denote by $\pi:\mathbb
P(W^\vee)\rightarrow\Spec K $ the canonical
morphism. Let $m$ be a strictly positive
integer which is divisible by all $r^{(i)}$
and $L$ be a $G$-linearized line bundle on
$\mathbb P(W^\vee)$ as follows:
\begin{equation}L:=\mathcal
O_{W^\vee}(m)\otimes\bigotimes_{i=1}^n
\pi^*(\det V^{(i)})^{\otimes
(m/r^{(i)})}.\end{equation}

For any rational point $x$ of $\mathbb
P(W^\vee)$, we define a function
$\Lambda_x:\mathbf{Fil}_{V^{(1)}}^{\mathbb Q
}\times\cdots\times\mathbf{Fil}_{V^{(n)}}^{\mathbb
Q}\rightarrow\mathbb R$ such that
\begin{equation}\label{Equ:Psi_x definition}
\Lambda_x(\mathcal G^{(1)},\cdots,\mathcal
G^{(n)})= \frac{\mathbb E[\mathcal
G^{(1)}]+\cdots+\mathbb E[\mathcal G^{(n)}]-
\lambda_{\mathcal G^{(1)}\otimes\cdots\otimes
\mathcal G^{(n)}}(v_x)}{(\|\mathcal
G^{(1)}\|^2+\cdots+\|\mathcal
G^{(n)}\|^2)^{\frac 12}}
\end{equation}
if at least one filtration among the
$\mathcal G^{(i)}$'s is non-trivial, and
$\Lambda_x(\mathcal G^{(1)},\cdots,\mathcal
G^{(n)})=0$ otherwise. We recall that in
\eqref{Equ:Psi_x definition}, $v_x$ is an
arbitrary non-zero element in $x$. Note that
the function $\Lambda_x$ is invariant by
dilation. In other words, for any positive
number $\varepsilon>0$,
\[\Lambda_x(\psi_\varepsilon\mathcal G^{(1)},\cdots,\psi_{\varepsilon}
\mathcal G^{(n)})=\Lambda_x(\mathcal
G^{(1)},\cdots,\mathcal G^{(n)}),\] where the
dilation $\psi_{\varepsilon}$ is defined in
\eqref{Equ:dilation definition}.

\begin{proposition}\label{Pro:Hilbert Mumford Totaro criterion}
Let $x$ be a rational point of $\mathbb
P(W^\vee)$. Then the point $x$ is {\bf not}
semistable for the action of $G$ relatively
to $L$ if and only if the function
$\Lambda_x$ defined above takes at least one
strictly negative value.
\end{proposition}
\begin{proof}
By Propositions \ref{Pro:Totaro 1} and
\ref{Totaro 2}, for any rational point $x$ of
$\mathbb P(W^\vee)$,
\begin{equation}\label{Equ:calcul de mu}\mu(x,h,L)=m\Big(\sum_{i=1}^n\mathbb E[\mathcal
F^{h_i}]- \lambda_{\mathcal
F^h}(v_x)\Big).\end{equation}

``$\Longrightarrow$'': By the Hilbert-Mumford
criterion (Theorem \ref{Thm:Hilbert-Mumford
criterion}), there exists a one-parameter
subgroup $h=(h_1,\cdots,h_n)$ of $G$ such
that $\mu(x,h,L)<0$. The filtration $\mathcal
F^h$ of $W$ associated with $h$ coincides
with the tensor product filtration $\mathcal
F^{h_1}\otimes\cdots\otimes\mathcal F^{h_n}$,
where $\mathcal F^{h_i}$ is the filtration of
$V^{(i)}$ associated with $h_i$. Therefore,
\[\Lambda_x(\mathcal F^{h_1},\cdots,\mathcal F^{h_n})
=\frac{\mu(x,h,L)}{m(\|\mathcal
F^{h_1}\|^2+\cdots+\|\mathcal
F^{h_n}\|^2)^{\frac 12}}<0.\]

``$\Longleftarrow$'': Suppose that $(\mathcal
G^{(1)},\cdots,\mathcal G^{(n)})$ is an
element in $\mathbf{Fil}_{V^{(1)}}^{\mathbb Q
}\times\cdots\times\mathbf{Fil}_{V^{(n)}}^{\mathbb
Q}$ such that $\Lambda_x(\mathcal
G^{(1)},\cdots,\mathcal G^{(n)})<0$. By
equalities \eqref{Equ:dilation},
\eqref{Equ:dilatation et produit tensoriel}
and the invariance of $\Lambda_x$ by
dilation, we can assume that $\mathcal
G^{(1)},\cdots,\mathcal G^{(n)}$ are all
supported by $\mathbb Z$. In this case, there
exists, for each $1\le i\le n$, a
one-parameter subgroup $h_i$ of
$\mathbb{GL}(V^{(i)})$ such that $\mathcal
F^{h_i}=\mathcal G^{(i)}$. Let
$h=(h_1,\cdots,h_n)$. By combining the
negativity of $\Lambda_x(\mathcal
F^{h_1},\cdots,\mathcal F^{h_n})$ with
\eqref{Equ:calcul de mu}, we obtain
$\mu(x,h,L)<0$, so $x$ is not semistable.
\end{proof}

Proposition \ref{Pro:Totaro96 generla} below
generalizes Proposition 2 of \cite{Totaro96}.
The proof uses Lemma \ref{Pro:lemma 3 of
Totaro}, which is equivalent to Lemma 3 of
\cite{Totaro96}, or Lemma 1.1 of
\cite{Ramanan_Ramanathan}. See
\cite{Ramanan_Ramanathan} for the proof of
the lemma.

\begin{lemma}
\label{Pro:lemma 3 of Totaro} Let $n\ge 1$ be
an integer and let $\mathscr T$ be a finite
non-empty family of linear forms on $\mathbb
R^n$. Let $\Lambda:\mathbb
R^n\rightarrow\mathbb R$ such that
$\Lambda(y)=\|y\|^{-1}
\displaystyle\max_{l\in\mathscr T}l(y)\text{
for }y\neq 0$, and that $\Lambda(0)=0$.
Suppose that the function $\Lambda$ takes at
least a strictly negative value. Then
\begin{enumerate}[1)]
\item the function $\Lambda$ attains its
minimum value, furthermore, all points in
$\mathbb R^n$ minimizing $\Lambda$ are
proportional;
\item if $c$ is the minimal value of
$\Lambda$ and if $y_0\in\mathbb R^n$ is a
minimizing point of $\Lambda$, then for any
$y\in\mathbb R^n$,
\begin{equation}\label{Equ:minoration de Lambda(y)}\Lambda(y)\ge c\frac{\langle
y_0,y\rangle}{\|y_0\|\cdot\|y\|};\end{equation}
\item if in addition all linear forms in $\mathscr T$ are
of rational coefficients, then there exists a
point in $\mathbb Z^n$ which minimizes
$\Lambda$.
\end{enumerate}
\end{lemma}

\begin{proposition}\label{Pro:Totaro96 generla}
With the notation of Proposition
\ref{Pro:Hilbert Mumford Totaro criterion},
if $x$ is not semistable for the action of
$G$ relatively to $L$, then the function
$\Lambda_x$ attains its minimal value.
Furthermore, the element in
$\mathbf{Fil}_{V^{(1)}}^{\mathbb Q
}\times\cdots\times\mathbf{Fil}_{V^{(n)}}^{\mathbb
Q}$ minimizing $\Lambda_x$ is unique up to
dilatation. Finally, if $(\mathcal
F^{(1)},\cdots,\mathcal F^{(n)})$ is an
element in $\mathbf{Fil}_{V^{(1)}}^{\mathbb Q
}\times\cdots\times\mathbf{Fil}_{V^{(n)}}^{\mathbb
Q}$ minimizing $\Lambda_x$ and if $c$ is the
minimal value of $\Lambda_x$, then for any
element $(\mathcal G^{(1)},\cdots,\mathcal
G^{(n)})$ in $\mathbf{Fil}_{V^{(1)}}^{\mathbb
Q
}\times\cdots\times\mathbf{Fil}_{V^{(n)}}^{\mathbb
Q}$, the following inequality holds:
\begin{equation}\label{Equ:estimation de Lambda x}
\sum_{i=1}^n\mathbb E[\mathcal
G^{(i)}]-\lambda_{\mathcal
G^{(1)}\otimes\cdots\otimes\mathcal
G^{(n)}}(v_x)\ge c\frac{\langle\mathcal
F^{(1)},\mathcal
G^{(1)}\rangle+\cdots+\langle\mathcal
F^{(n)},\mathcal G^{(n)}\rangle}{(\|\mathcal
F^{(1)}\|^2+\cdots+\|\mathcal
F^{(n)}\|^2)^{\frac 12}}
\end{equation}
\end{proposition}
\begin{proof}
For each integer $1\le i\le n$, let
$\mathbf{e}^{(i)}=(e^{(i)}_j)_{1\le j\le
r^{(i)}}$ be a basis of $V^{(i)}$. Let
${\mathbf{e}}=(\mathbf{e}^{(i)})_{1\le i\le
n}$. Denote by $\Lambda_x^{{\mathbf e}}$ the
restriction of $\Lambda_x$ on
$\mathbf{Fil}_{\mathbf{e}^{(1)}}^{\mathbb Q
}\times\cdots\times\mathbf{Fil}_{\mathbf{e}^{(n)}}^{\mathbb
Q}$. The space
$\mathbf{Fil}_{\mathbf{e}^{(1)}}^{\mathbb Q
}\times\cdots\times\mathbf{Fil}_{\mathbf{e}^{(n)}}^{\mathbb
Q}$ is canonically embedded in
$\mathbf{Fil}_{\mathbf{e}^{(1)}}\times\cdots\times\mathbf{Fil}_{\mathbf{e}^{(n)}}$,
which can be identified with $\mathbb
R^{r^{(1)}}\times\cdots\times\mathbb
R^{r^{(n)}}$ through
$\Phi_{\mathbf{e}^{(1)}}\times\cdots\Phi_{\mathbf{e}^{(n)}}$
(see Proposition \ref{Pro:euclidean space of
filtrations}). We extend natually
$\Lambda_x^{{\mathbf{e}}}$ to a function
$\Lambda_x^{{\mathbf{e}},\dagger}$ on
$\mathbf{Fil}_{\mathbf{e}^{(1)}}\times\cdots\times\mathbf{Fil}_{\mathbf{e}^{(n)}}$,
whose numerator part is the maximal value of
a finite number of linear forms with rational
coefficients (see Proposition \ref{Pro:lambda
mathcal F as inf des formel ineares}) and
whose denominator part is just the norm of
vector in the Euclidean space. Then by Lemma
\ref{Pro:lemma 3 of Totaro}, the function
$\Lambda_x^{\mathbf{e},\dagger}$ attains its
minimal value, and there exists an element in
$\mathbf{Fil}_{\mathbf{e}^{(1)}}^{\mathbb Q
}\times\cdots\times\mathbf{Fil}_{\mathbf{e}^{(n)}}^{\mathbb
Q}$ which minimizes
$\Lambda_x^{\mathbf{e},\dagger}$. By
definition the same element also minimizes
$\Lambda_x^{\mathbf{e}}$. Since the function
$\Lambda_x^{\mathbf{e}}$, viewed as a
function on $\mathbb
R^{r^{(1)}+\cdots+r^{(n)}}$, only depends on
the set
\[\Big\{S\subset\prod_{i=1}^n\{1,\cdots,r^{(i)}\}\;\Big|\;
v_x\in\sum_{(j_1,\cdots,j_n)\in
S}Ke^{(1)}_{j_1}\otimes\cdots\otimes
e^{(n)}_{j_n} \Big\}.\] Therefore, there are
only a finite number of functions on
Euclidean space of dimension
$r^{(1)}+\cdots+r^{(n)}$ of the form
$\Lambda_x^{\mathbf{e}}$. Thus we deduce that
the function $\Lambda_x$ attains globally its
minimal value, and the minimizing element of
$\Lambda_x$ could be chosen in
$\mathbf{Fil}_{V^{(1)}}^{\mathbb Q
}\times\cdots\times\mathbf{Fil}_{V^{(n)}}^{\mathbb
Q}$.

Suppose that there are two elements in
$\mathbf{Fil}_{V^{(1)}}^{\mathbb Q
}\times\cdots\times\mathbf{Fil}_{V^{(n)}}^{\mathbb
Q}$ which minimizes $\Lambda_x$. By Bruhat's
decomposition, we can choose $\mathbf{e}$ as
above such that both elements lie in
$\mathbf{Fil}_{\mathbf{e}^{(1)}}^{\mathbb Q
}\times\cdots\times\mathbf{Fil}_{\mathbf{e}^{(n)}}^{\mathbb
Q}$. Therefore, by Lemma \ref{Pro:lemma 3 of
Totaro} they differ only by a dilation.
Finally to prove inequality
\eqref{Equ:estimation de Lambda x}, it
suffices to choose $\mathbf{e}$ such that
$(\mathcal F^{(1)},\cdots,\mathcal F^{(n)})$
and $(\mathcal G^{(1)},\cdots,\mathcal
G^{(n)})$ are both in
$\mathbf{Fil}_{\mathbf{e}^{(1)}}^{\mathbb Q
}\times\cdots\times\mathbf{Fil}_{\mathbf{e}^{(n)}}^{\mathbb
Q}$, and then apply Lemma \ref{Pro:lemma 3 of
Totaro} 2).
\end{proof}

Although the minimizing filtrations
$(\mathcal F^{(1)},\cdots,\mathcal F^{(n)}) $
in Proposition \ref{Pro:Totaro96 generla} are
{\it a priori} supported by $\mathbb Q$, it
is always possible to choose them to be
supported by $\mathbb Z$ after a dilation.

In the rest of the section, let $x$ be a
rational point of $\mathbb P(W^\vee)$ which
is {\bf not} semistable for the action of $G$
relatively to $L$. We fix an element
$(\mathcal F^{(1)},\cdots,\mathcal F^{(n)})$
in $\mathbf{Fil}_{V^{(1)}}^{\mathbb Z
}\times\cdots\times\mathbf{Fil}_{V^{(n)}}^{\mathbb
Z}$ minimizing $\Lambda_x$. Define
\begin{equation}\label{Equ:widetilde c}
\displaystyle\widetilde
c:=\frac{c}{(\|\mathcal
F^{(1)}\|^2+\cdots+\|\mathcal
F^{(n)}\|^2)^{\frac 12}}.
\end{equation}
Note that $\widetilde c<0$. Moreover, it is a
rational number since the following equality
holds:
\[\widetilde c=\frac{\Lambda_x(\mathcal
F^{(1)},\cdots,\mathcal F^{(n)})}{(\|\mathcal
F^{(1)}\|^2+\cdots+\|\mathcal
F^{(n)}\|^2)^{\frac 12}}=\frac{\mathbb
E[\mathcal F^{(1)}]+\cdots+\mathbb E[\mathcal
F^{(n)}]- \lambda_{\mathcal
F^{(1)}\otimes\cdots\otimes \mathcal
F^{(n)}}(v_x)}{\|\mathcal
F^{(1)}\|^2+\cdots+\|\mathcal F^{(n)}\|^2}.\]

We suppose that $\mathcal F^{(i)}$
corresponds to the flag
\[\mathscr D^{(i)}:V^{(i)}=V_0^{(i)}\supsetneq V_1^{(i)}\supsetneq\cdots\supsetneq
V_{d^{(i)}}^{(i)}=0\] and the strictly
increasing sequence of integers
$\lambda^{(i)}=(\lambda_j^{(i)})_{0\le j<
d^{(i)}}$. Let $\widetilde G$ be the
algebraic group
\[\widetilde G:=\prod_{i=1}^n\prod_{j=0}^{d^{(i)}-1}\mathbb{GL}(
V_{j}^{(i)}/V_{j+1}^{(i)}).\] Let $\mathcal
F=\mathcal
F^{(1)}\otimes\cdots\otimes\mathcal F^{(n)}$
and $\beta=\lambda_{\mathcal F}(v_x)$, which
is the largest integer $i$ such that
$v_x\in\mathcal F_iW$. Let $\widetilde
W:=\mathcal F_iW/\mathcal F_{i+1}W$ and let
$\widetilde v_x$ be the canonical image of
$v_x$ in $\widetilde W$. Notice that
\[\widetilde W=\sum_{\lambda_{j_1}^{(1)}+\cdots+\lambda_{j_n}^{(n)}\ge\beta}
\bigotimes_{i=1}^nV^{(i)}_{j_i}\Bigg/\sum_{\lambda_{j_1}^{(1)}+\cdots+
\lambda_{j_n}^{(n)}>\beta}
\bigotimes_{i=1}^nV^{(i)}_{j_i}\cong\bigoplus_{\lambda_{j_1}^{(1)}+
\cdots+\lambda_{j_n}^{(n)}=\beta}\bigotimes_{i=1}^n
\Big(V^{(i)}_{j_i}/V^{(i)}_{j_i+1}\Big).
\]
So the algebraic group $\widetilde G$ acts
naturally on $\widetilde W$. Let $\widetilde
x$ be the rational point of $\mathbb
P(\widetilde W^\vee)$ corresponding to the
subspace of $\widetilde W$ generated by
$\widetilde v_x$.

For all integers $i,j$ such that $1\le i\le
n$ and $0\le j< d^{(i)}$, let $r_j^{(i)}$ be
the rank of $V_j^{(i)}/V_{j+1}^{(i)}$ over
$K$. We choose a strictly positive integer
$N$ divisible by all $r^{(i)}=\rang_KV^{(i)}$
and such that, for any integers $i$ and $j$ satisfying $1\le i\le n$
and $0\le j< d^{(i)}$, the number
\[\displaystyle a_j^{(i)}:=-\frac{N\widetilde
c\lambda_j^{(i)}}{r^{(i)}}\] is an integer.
This is always possible since $\widetilde
c\in\mathbb Q$. The sequence
$(\lambda^{(i)}_j)_{0\le j<d^{(i)}}$ is
strictly increasing, so is
$\mathbf{a}^{(i)}:=(a_j^{(i)})_{0\le j<
d^{(i)}}$. Finally we define $\displaystyle
b_j^{(i)}:=\frac{N}{r^{(i)}}+a_j^{(i)}$.

We are now able to establish an explicit
version of Proposition 1.12 in
\cite{Ramanan_Ramanathan} for product of
general linear groups.

\begin{proposition}\label{Thm:theorem de Ramanan-Ramanathan}
Let $\widetilde{\pi}:\mathbb P(\widetilde
W^\vee)\rightarrow\Spec K$ be the canonical
morphism and let
\[\widetilde L:=\mathcal O_{\widetilde W^\vee}(N)\otimes\Big(
\bigotimes_{i=1}^n\bigotimes_{j=0}^{d^{(i)}-1}
\widetilde\pi^*\big(\Lambda^{r_j^{(i)}}(V_j^{(i)}/V_{j+1}^{(i)})\big)^{\otimes
b_j^{(i)}} \Big).\] Then the rational point
$\widetilde x$ of $\mathbb P(\widetilde
W^\vee)$ is semistable for the action of
$\widetilde G$ relatively to the
$G$-linearized line bundle $\widetilde{L}$.
\end{proposition}
\begin{proof}
For any integers $i$ and $j$ such that $1\le
i\le n$ and $0\le j< d^{(i)}$, we choose an
arbitrary filtration $\mathcal G^{(i),j}$ of
$V_j^{(i)}/V_{j+1}^{(i)}$ supported by
$\mathbb Z$. We have explained in Subsection
\ref{Subsec:Construction of filtration from
subquotients} how to construct a new
filtration $\mathcal G^{(i)}$ of $V^{(i)}$
from $\mathcal G^{(i),j}$. Let
\[\mathcal G=\bigotimes_{i=1}^n\mathcal G^{(i)},\qquad
\widetilde{\mathcal
G}=\bigoplus_{\lambda_{j_1}^{(1)}+\cdots+\lambda_{j_n}^{(n)}=\beta}
\bigotimes_{i=1}^n\mathcal G^{(i),j_i}.\]
From the construction we know that
$\lambda_{\mathcal
G}(v_x)=\lambda_{\widetilde{\mathcal
G}}(\widetilde v_x)$.  Using
\eqref{Equ:esperance de filtration
sousquotient}, the inequality
\eqref{Equ:estimation de Lambda x} implies:
\begin{equation}\label{Equ:subquotient stability}
\sum_{i=1}^n\sum_{j=0}^{d^{(i)}-1}\frac{r_j^{(i)}}{r^{(i)}}
\mathbb E[\mathcal
G^{(i),j}]-\sum_{i=1}^n\sum_{j=0}^{d^{(i)}-1}\frac{\widetilde
c\lambda_j^{(i)}r_j^{(i)}}{r^{(i)}} \mathbb
E[\mathcal
G^{(i),j}]-\lambda_{\widetilde{\mathcal
G}}(\widetilde v_x)\ge 0,
\end{equation}
where the constant $\widetilde c$ is defined
in \eqref{Equ:widetilde c}. Hence
\begin{equation}\label{Equ:semistbaility de xtilde}
\sum_{i=1}^n\sum_{j=0}^{d^{(i)}-1}
b_j^{(i)}r_j^{(i)}\mathbb E[\mathcal
G^{(i),j}]-N\lambda_{\widetilde {\mathcal
G}}(\widetilde v_x)\ge 0.\end{equation}

Let $h$ be an arbitrary one-parameter
subgroup of $\widetilde G$ corresponding to
filtrations $\mathcal G^{(i),j}$. By
Propositions \ref{Pro:Totaro 1} and
\ref{Totaro 2}, together with the fact that
$\mu(\widetilde x,h,\cdot)$ is a homomorphism
of groups, we obtain
\[\begin{split}
\mu(\widetilde x,h,\widetilde
L)&=\mu(\widetilde x,h,\mathcal O_{\widetilde
W^\vee}(N))+
\sum_{i=1}^n\sum_{j=0}^{d^{(i)}-1}
b_j^{(i)}r_j^{(i)}
\mathbb E[\mathcal G^{(i),j}]\\
&=-N\lambda_{\widetilde{\mathcal
G}}(\widetilde
v_x)+\sum_{i=1}^n\sum_{j=0}^{d^{(i)}-1}
b_j^{(i)}r_j^{(i)} \mathbb E[\mathcal
G^{(i),j}]\ge 0.
\end{split}\]
By Hilbert-Mumford criterion, the point
$\widetilde x$ is semistable for the action
of $\widetilde G$ relatively to $\widetilde
L$.
\end{proof}

Finally we point out the following
consequence of the inequality
\eqref{Equ:semistbaility de xtilde}.

\begin{proposition}\label{Pro:filtration minimisant est desperance 0}
The minimizing filtrations $(\mathcal
F^{(1)},\cdots,\mathcal F^{(n)})$ satisfy
\[\mathbb E[\mathcal F^{(1)}]=\cdots=\mathbb E[\mathcal F^{(n)}]=0.\]
In other words, the equality
$\sum_{j=0}^{d^{(i)}-1}a_j^{(i)}r_j^{(i)}=0$
holds, or equivalently,
$\sum_{j=0}^{d^{(i)}-1}\lambda_j^{(i)}r_j^{(i)}=0$
for any $i\in\{1,\cdots,n\}$.
\end{proposition}
\begin{proof}
Let $(u_i)_{1\le i\le n}$ be an arbitrary
sequence of integers. For all integers $i,j$
such that $1\le i\le n$ and $0\le j<d^{(i)}$,
let $\mathcal G^{(i),j}$ be the filtration of
$V_j^{(i)}/V_j^{(i+1)}$ which is supported by
$\{u_i\}$. Note that in this case
$\widetilde{\mathcal G}$ is supported by
$\{u_1+\cdots+u_n\}$. The inequality
\eqref{Equ:semistbaility de xtilde} gives
\[\sum_{i=1}^n\sum_{j=0}^{d^{(i)}-1}
b_j^{(i)}r_j^{(i)}u_i-N\sum_{i=1}^nu_i=\sum_{i=1}^n
u_i\sum_{j=0}^{d^{(i)}-1}
a_j^{(i)}r_j^{(i)}\ge 0.\] Since $(u_i)_{1\le
i\le n}$ is arbitrary, we obtain
$\sum_{j=0}^{d^{(i)}-1}a_j^{(i)}r_j^{(i)}=0$,
and therefore
$\sum_{j=0}^{d^{(i)}-1}\lambda_j^{(i)}r_j^{(i)}=0$.
\end{proof}

\section{A criterion of Arakelov semistability for Hermitian vector bundles}
\label{Sec:criterion of semistability}

\hskip\parindent We shall give a
semistability criterion for Hermitian vector
bundles, which is the arithmetic analogue of
a result due to Bogomolov in geometric
framework (see \cite{Raynaud81}).

Let $\overline E$ be a non-zero Hermitian
vector bundle over $\Spec\mathcal O_K$ and
let $V=E_K$. We denote by $r$ its rank. If
$\mathscr D:V=V_0\supsetneq
V_1\supsetneq\cdots\supsetneq V_d=0$ is a
flag of $V$, it induces a strictly decreasing
sequence of saturated sub-$\mathcal
O_K$-modules $E=E_0\supsetneq
E_1\supsetneq\cdots\supsetneq E_d=0$ of $E$.
For any integer $j$ such that $0\le j<d$, let
$r_j$ be the rank of $E_j/E_{j+1}$. If
$\mathbf{a}=(a_j)_{0\le j<d}$ is an element
in $r\mathbb Z^d$, we denote by
$\overline{\mathscr L}_{\hskip -2 pt\mathscr
D}^{\mathbf{a}}$ the Hermitian line bundle on
$\Spec\mathcal O_K$ as follows
\begin{equation}\label{Equ:L D a}
\overline{\mathscr L}_{\hskip -2
pt\mathscr D}^{\mathbf{a}}:=
\bigotimes_{j=0}^{d-1}\Big((\Lambda^{r_j}(\overline
E_j/\overline E_{j+1}))^{\otimes
a_j}\otimes(\Lambda^r\overline
E)^{\vee\otimes\frac{r_ja_j}{r}}\Big).\end{equation}
If $\mathbf{a}=(a_j)_{0\le j<d}\in\mathbb
Z^d$ satisfies $\sum_{j=0}^{d-1}r_ja_j=0$, we
define $\overline{\mathscr L}_{\hskip -2
pt\mathscr D}^{\mathbf{a}}:=
\bigotimes_{j=0}^{d-1}(\Lambda^{r_j}(\overline
E_j/\overline E_{j+1}))^{\otimes a_j}$.

\begin{proposition}\label{Pro:critere desemistabaility}
If the Hermitian vector bundle $\overline E$
is semistable (resp. stable), then for any
integer $d\ge 1$, any flag $\mathscr D$ of
length $d$ of $V$, and any strictly
increasing sequence $\mathbf{a}=(a_j)_{0\le
j<d}$ of integers either in $r\mathbb Z^d$,
or such that $\sum_{j=0}^{d-1}r_ja_j=0$, we
have $\widehat{\deg}(\overline{\mathscr
L}^{\mathbf{a}}_{\hskip -2pt\mathscr D})\le
0$ (resp. $\widehat{\deg}(\overline{\mathscr
L}^{\mathbf{a}}_{\hskip -2pt\mathscr D})<
0$).
\end{proposition}
\begin{proof}
By definition,
\[\begin{split}
\widehat{\deg}(\overline{\mathscr L}_{\hskip
-2 pt\mathscr D}^{\mathbf{a}})
&=\sum_{j=0}^{d-1}a_j\left[-\frac{\rang(E_{j})
-\rang(E_{j+1})}{r}\widehat{\deg}(\overline E
)+\widehat{\deg}(\overline{E}_{j})-\widehat{\deg}(\overline
E_{j+1} )\right]\\
&=\sum_{j=0}^{d-1}a_j\bigg[\rang(E_{j})\Big(\widehat{\mu}(
\overline E_{j})-\widehat{\mu}(\overline E
)\Big)-\rang(E_{j+1})\Big(\widehat{\mu}(\overline
E_{j+1})-
\widehat{\mu}(\overline E)\Big)\bigg]\\
&=\sum_{j=1}^{d-1}(a_{j}-a_{j-1})\rang(E_j)\big(
\widehat{\mu}(\overline
E_j)-\widehat{\mu}(\overline E)\big).
\end{split}\]
If $\overline E$ is semistable (resp.
stable), then for any integer $j$ such that
$1\le j<d$, we have $\widehat{\mu}(\overline
E_j)\le\widehat{\mu}(\overline E)$ (resp.
$\widehat{\mu}(\overline
E_j)<\widehat{\mu}(\overline E)$). Hence
$\widehat{\deg}(\overline{\mathscr
L}^{\mathbf{a}}_{\hskip -2pt\mathscr D}))\le
0$ (resp. $\widehat{\deg}(\overline{\mathscr
L}^{\mathbf{a}}_{\hskip -2pt\mathscr
D}))<0$).
\end{proof}

\begin{remark}
The converse of Proposition \ref{Pro:critere
desemistabaility} is also true. Let $E_1$ be
a saturated sub-$\mathcal O_K$-module of $E$.
Consider the flag $\mathscr D:V\supsetneq
E_{1,K}\supsetneq 0$ and the integer sequence
$\mathbf{a}=(0,r)$. Then
$\widehat{\deg}(\overline{\mathscr L}
^{\mathbf{a}}_{\hskip -2 pt\mathscr{D}})
=r\rang(E_1)\big(\widehat{\mu}(\overline
E)-\widehat{\mu}(\overline E_1)\big)$.
Therefore $\widehat{\mu}(\overline
E_1)\le\widehat{\mu}(\overline E)$ (resp.
$\widehat{\mu}(\overline
E_1)<\widehat{\mu}(\overline E)$). Since
$E_1$ is arbitrary, the Hermitian vector
bundle $\overline E$ is semistable (resp.
stable).
\end{remark}

\section{Upper bound for the degree of a Hermitian line subbundle}\label{Sec:upper bound for the degree}

\hskip\parindent In this section, we shall
give an upper bound for the Arakelov degree
of a Hermitian line subbundle of a finite
tensor product of Hermitian vector bundles.
As explained in Section
\ref{Sec:Introduction}, we shall use the
results established in Section \ref{Sec:More
facts in geometric invariant theory}  to
reduce our problem to the case with
semistability condition (in geometric
invariant theory sense), which has already
been discussed in Section \ref{Sec:Upper
bound with hypothesis of semi-stability }. We
point out that, in order to obtain the same
estimation as \eqref{Equ:subline bundle
semistable} in full generality, we should
assume that all Hermitian vector bundles
$\overline E_i$ are semistable, as a price
paid for removing the semistability condition
for $M_K$.

We denote by $K$ a number field and by
$\mathcal O_K$ its integer ring. Let
$(\overline E^{(i)})_{1\le i\le n}$ be a
family of {\bf semistable} Hermitian vector
bundles on $\Spec\mathcal O_K$. For any
$i\in\{1,\cdots,n\}$, let $r^{(i)}$ be the
rank of $E^{(i)}$ and $V^{(i)}=E^{(i)}_K$.
Let $\overline E=\overline
E^{(1)}\otimes\cdots\otimes\overline E^{(n)}$
and $W=E_K$. We denote by $\pi:\mathbb
P(W^\vee)\rightarrow\Spec K$ the natural
morphism. The algebraic group
$G:=\mathbb{GL}_K(V^{(1)})\times_K\cdots\times_K\mathbb{GL}_K(V^{(n)})$
acts naturally on $\mathbb P(W^\vee)$. Let
$\overline M$ be a Hermitian line subbundle
of $\overline E$ and $m$ be a strictly
positive integer which is divisible by all
$r^{(i)}$'s.

\begin{proposition}\label{Pro:majoration of upper degree}
For any Hermitian line subbundle $\overline
M$ of $\overline
E^{(1)}\otimes\cdots\otimes\overline
E^{(n)}$, we have
\[\widehat{\deg}(\overline M)\le\sum_{i=1}^n\Big(
\widehat{\mu}(\overline E^{(i)})+\frac
12\log(\rang E^{(i)})\Big).\]
\end{proposition}
\begin{proof}
We have proved that if
$M_K$ is semistable for the action of $G$
relatively to $\mathcal
O_{W^\vee}(m)\otimes\pi^*\Big(\bigotimes_{i=1}^n
(\Lambda^{r^{(i)}}V^{(i)})^{\otimes
m/r^{(i)}}\Big)$, where $m$ is a strictly positive integer which is divisible by all $r^{(i)}$, then the following
inequality holds:
\[\widehat{\deg}(\overline M)\le\sum_{i=1}^n\Big(\widehat{\mu}(
\overline E_i)+\frac 12\log r^{(i)}\Big).\]

If this hypothesis of semistability is not
fulfilled, by Proposition \ref{Thm:theorem de
Ramanan-Ramanathan}, there exist two strictly
positive integers $N$ and $\beta$, and for
any $i\in\{1,\cdots,n\}$,
\begin{enumerate}[1)]
\item a flag
\[\mathscr D^{(i)}:V^{(i)}=V_0^{(i)}\supsetneq V_1^{(i)}
\supsetneq\cdots\supsetneq V_{d^{(i)}}^{(i)}=0\] of
$V^{(i)}$ corresponding to the sequence
\[E^{(i)}=E_0^{(i)}\supsetneq E_1^{(i)}
\supsetneq\cdots\supsetneq E_{d^{(i)}}^{(i)}=0\] of
saturated sub-$\mathcal O_K$-modules of $E$,
\item two strictly increasing sequence $\lambda^{(i)}=
(\lambda^{(i)}_j)_{0\le j< d^{(i)}}$ and
$\mathbf{a}^{(i)}=(a_j^{(i)})_{0\le j<d^{(i)}}$ of
integers,
\end{enumerate}
such that
\begin{enumerate}[i)]
\item $N$ is divisible by all $r^{(i)}$'s,
\item for any integer $i$ such that $1\le i\le n$,
$\sum_{j=0}^{d^{(i)}-1}a_j^{(i)}r_j^{(i)}=0$,
where
$r_j^{(i)}=\rang(V_j^{(i)}/V_{j+1}^{(i)})$,
\item the inclusion of $M$ in $E$ factorizes through
$\displaystyle\sum_{\lambda_{i_1}^{(1)}+\cdots\lambda_{i_n}^{(n)}\ge
\beta}E_{i_1}^{(1)}\otimes\cdots\otimes
E_{i_n}^{(n)}$,
\item the canonical image of $M_K$ in
\[\widetilde W:=\sum_{\lambda_{j_1}^{(1)}+\cdots+\lambda_{j_n}^{(n)}\ge\beta}
\bigotimes_{i=1}^nV^{(i)}_{j_i}\Bigg/\sum_{\lambda_{j_1}^{(1)}+\cdots+
\lambda_{j_n}^{(n)}>\beta}
\bigotimes_{i=1}^nV^{(i)}_{j_i}\cong\bigoplus_{\lambda_{j_1}^{(1)}+
\cdots+\lambda_{j_n}^{(n)}=\beta}\bigotimes_{i=1}^n
\Big(V^{(i)}_{j_i}/V^{(i)}_{j_i+1}\Big).
\]is non-zero, and is semistable for the action of the
group \[\widetilde
G:=\prod_{i=1}^n\prod_{j=0}^{d^{(i)}-1}\mathbb{GL}(
V_{j}^{(i)}/V_{j+1}^{(i)})\] relatively to
\[\mathcal O_{\widetilde W^\vee}(N)\otimes\Big(
\bigotimes_{i=1}^n\bigotimes_{j=0}^{d^{(i)}-1}
\widetilde\pi^*\big(\Lambda^{r_j^{(i)}}(V_j^{(i)}/V_{j+1}^{(i)})\big)^{\otimes
b_j^{(i)} }\Big),\] where
$\widetilde\pi:\mathbb P(\widetilde
W^\vee)\rightarrow\Spec K$ is the canonical
morphism, and
$b_j^{(i)}=N/r^{(i)}+a_j^{(i)}$.
\end{enumerate}
Note that $\bigotimes_{j=0}^{d^{(i)}-1}\big(
\Lambda^{r_j^{(i)}}(\overline
E_j^{(i)}/\overline
E_{j+1}^{(i)})\big)^{\otimes a_j^{(i)}}$ is
nothing other than $\overline{\mathscr
L}_{\hskip -2 pt\mathscr
D^{(i)}}^{\mathbf{a}^{(i)}}$ defined in
\eqref{Equ:L D a}.

Applying Theorem \ref{Thm:majoration de degre
d'Arakelov avec l'hypo de semistability}, we
get
\[\begin{split}\widehat{\deg}(\overline M)&\le\frac{1}{N}
\sum_{i=1}^n\sum_{j=0}^{d^{(i)}-1}\frac{N}{r^{(i)}}\big(\widehat{\deg}
(\overline
E_j^{(i)})-\widehat{\deg}(\overline
E_{j+1}^{(i)})\big)+\frac{1}{N}\sum_{i=1}^n\widehat{\deg}
\overline{\mathscr L}_{\hskip -2 pt\mathscr
D^{(i)}}^{\mathbf{a}^{(i)}}+\sum_{i=1}^n\sum_{j=0}^{d^{(i)}-1}
\frac{r_j^{(i)}b_j^{(i)}}{2N}\log r_j^{(i)}\\
&=\sum_{i=1}^n\widehat{\mu}(\overline
E^{(i)})+\frac{1}{N}\sum_{i=1}^n\widehat{\deg}
\overline{\mathscr L}_{\hskip -2 pt\mathscr
D^{(i)}}^{\mathbf{a}^{(i)}}+\sum_{i=1}^n\sum_{j=0}^{d^{(i)}-1}
\frac{r_j^{(i)}b_j^{(i)}}{2N}\log r_j^{(i)}\\
&\le \sum_{i=1}^n\widehat{\mu}(\overline
E^{(i)})+\sum_{i=1}^n\sum_{j=0}^{d^{(i)}-1}
\frac{r_j^{(i)}b_j^{(i)}}{2N}\log r_j^{(i)},
\end{split}\]
where the last inequality is because
$\overline E^{(i)}$'s are Arakelov semistable
(see Proposition \ref{Pro:critere
desemistabaility}). By Theorem
\ref{Thm:semistability et determinant},The
semistability of the canonical image of $M_K$
implies that $b_j^{(i)}\ge 0$. Therefore
\[\widehat{\deg}(\overline M)\le
\sum_{i=1}^n\widehat{\mu}(\overline
E^{(i)})+\sum_{i=1}^n\sum_{j=0}^{d^{(i)}-1}
\frac{r_j^{(i)}b_j^{(i)}}{2N}\log r^{(i)}.
\]
Since
$\sum_{j=0}^{d^{(i)}-1}r_j^{(i)}a_j^{(i)}=0$
for any integer $i$ such that $1\le i\le n$ (see
Proposition \ref{Pro:filtration minimisant
est desperance 0}), we have proved the
proposition.
\end{proof}

\begin{corollary}\label{Cor:majoration de max slp}
The following inequality is verified:
\begin{equation}\label{Equ:controle de pente maxi
sous cond de
semi}\widehat{\mu}_{\max}(\overline
E^{(1)}\otimes\cdots \otimes\overline
E^{(n)})\le\sum_{i=1}^n\Big(\widehat{\mu}(\overline
E^{(i)})+\log(\rang
E^{(i)})\Big)+\frac{\log|\Delta_K|}{2[K:\mathbb
Q]}.\end{equation}
\end{corollary}
\begin{proof}
Since the Hermitian line bundle $\overline M$
in Proposition \ref{Pro:majoration of upper
degree} is arbitrary, we obtain
\[\mathrm{u}\widehat{\deg}_n(\overline E^{(1)}
\otimes\cdots\otimes\overline E^{(n)})\le
\sum_{i=1}^n\Big(\widehat{\mu}(\overline
E^{(i)})+\frac 12\log(\rang E^{(i)})\Big).\]
Combining with \eqref{Equ:Bost-Kunnemann} we
obtain \eqref{Equ:controle de pente maxi sous
cond de semi}.
\end{proof}

\section{Proof of Theorem \ref{Thm:main theorem}}

We finally give the proof of Theorem
\ref{Thm:main theorem}.

\begin{lemma}\label{Lem:lemma pour thm principal}
Let $K$ be a number field and $\mathcal O_K$
be its integer ring. Let $(\overline
E_i)_{1\le i\le n}$ be a finite family of
non-zero Hermitian vector bundles
(non-necessarily semistable) and $\overline
E=\overline E_1\otimes\cdots\otimes\overline
E_n$. Then the following inequality holds:
\[\widehat{\mu}_{\max}(\overline E)\le
\sum_{i=1}^n\Big(\widehat{\mu}_{\max} (E_i)+\log(\rang
E_i)\Big)+\frac{\log|\Delta_K|}{2[K:\mathbb Q]}.\]
\end{lemma}
\begin{proof}
Let $F$ be a sub-$\mathcal O_K$-module of
$E$. By taking Harder-Narasimhan flags of
$E_i$'s (cf. \cite{BostBour96}), there
exists, for any $i$ such that $1\le i\le n$, a semistable
subquotient $\overline F_i/\overline G_i$ of
$E_i$ such that
\begin{enumerate}[1)]
\item $\widehat{\mu}(\overline F_i/\overline G_i)
\le\widehat{\mu}_{\max}(\overline E_i)$,
\item the inclusion homomorphism from $F$ to $E$
factorises through $F_1\otimes\cdots\otimes
F_n$,
\item the canonical image of $F$ in $(F_1/G_1)\otimes\cdots\otimes(F_n/
G_n)$ does not vanish.
\end{enumerate}
Combining with the slope inequality
\eqref{Equ:ineqality de pentes eux},
Corollary \ref{Cor:majoration de max slp}
implies that
\[\begin{split}
\widehat{\mu}_{\min}(\overline
F)&\le\sum_{i=1}^n\Big(\widehat{\mu}(\overline
F_i/\overline
G_i)+\log(\rang(F_i/G_i))\Big)+\frac{\log|\Delta_K|}{2[K:\mathbb
Q]}\\
&\le\sum_{i=1}^n\Big(\widehat{\mu}_{\max}(\overline
E_i)+\log(\rang
E_i)\Big)+\frac{\log|\Delta_K|}{2[K:\mathbb Q]}.
\end{split}\]
Since $F$ is arbitrary, the proposition is proved.
\end{proof}

{\noindent\bf Proof of Theorem \ref{Thm:main
theorem}} Let $N\ge 1$ be an arbitrary
integer. On one hand, by Lemma \ref{Lem:lemma
pour thm principal}, we have, by considering
$\overline E^{\otimes N}$ as
$\underbrace{\overline
E_1\otimes\cdots\otimes\overline
E_1}_{N\text{
copies}}\otimes\cdots\otimes\underbrace{\overline
E_n\otimes\cdots\otimes\overline
E_n}_{N\text{ copies}}$, that
\begin{equation*}\label{Equ:avant le dernier etape}\widehat{\mu}_{\max}(\overline E^{\otimes N})
\le\sum_{i=1}^n
N\Big(\widehat{\mu}_{\max}(\overline
E_i)+\log(\rang
E_i)\Big)+\frac{\log|\Delta_K|}{2[K:\mathbb
Q]}.\end{equation*} On the other hand, by
Corollary \ref{Cor:minormation de maximal
pent of tensor produc},
$\widehat{\mu}_{\max}(\overline E^{\otimes
N})\ge N\widehat{\mu}_{\max}(\overline E)$.
Hence
\[\widehat{\mu}_{\max}(\overline E)\le\sum_{i=1}^n\Big(
\widehat{\mu}_{\max}(\overline
E_i)+\log(\rang
E_i)\Big)+\frac{\log|\Delta_K|}{2N[K:\mathbb
Q]}.\] Since $N$ is arbitrary, we obtain by
taking $N\rightarrow+\infty$,
\[\widehat{\mu}_{\max}(\overline E)\le\sum_{i=1}^n\Big(
\widehat{\mu}_{\max}(\overline
E_i)+\log(\rang E_i)\Big),\] which completes
the proof.


\bibliography{chen}
\bibliographystyle{alpha}

%

\end{document}